\documentclass[11pt,reqno]{amsart}

\usepackage[T1]{fontenc}
\usepackage{lmodern}
\usepackage{microtype}
\usepackage{amsmath,amssymb,mathtools}
\usepackage{xcolor}
\usepackage{array,booktabs,float}
\usepackage{enumitem}
\usepackage[margin=1.10in]{geometry}
\usepackage[colorlinks=true,linkcolor=blue!55!black,citecolor=blue!55!black,urlcolor=blue!55!black]{hyperref}

\newtheorem{theorem}{Theorem}[section]
\newtheorem{proposition}[theorem]{Proposition}
\newtheorem{lemma}[theorem]{Lemma}
\newtheorem{corollary}[theorem]{Corollary}
\theoremstyle{definition}
\newtheorem{definition}[theorem]{Definition}

\newtheorem{remark}[theorem]{Remark}

\newcommand{\kk}{\Bbbk}
\newcommand{\Sym}{\operatorname{Sym}}
\newcommand{\Ann}{\operatorname{Ann}}
\newcommand{\Soc}{\operatorname{Soc}}
\newcommand{\im}{\operatorname{im}}
\newcommand{\Aut}{\operatorname{Aut}}
\newcommand{\Hom}{\operatorname{Hom}}
\newcommand{\GL}{\operatorname{GL}}
\newcommand{\PGL}{\operatorname{PGL}}
\newcommand{\ad}{\operatorname{ad}}
\newcommand{\rank}{\operatorname{rank}}
\newcommand{\Ext}{\operatorname{Ext}}
\newcommand{\Der}{\operatorname{Der}}
\newcommand{\cL}{\mathcal L}
\newcommand{\cZ}{\mathcal Z}
\newcommand{\fg}{\mathfrak g}
\newcommand{\fh}{\mathfrak h}
\newcommand{\NR}{\mathrm{NR}}

\title[Quot-stack moduli and transverse deformations]
{Quot-Stack Moduli and Transverse Deformations\\
of Graded Metabelian Lie Algebras}
\author{Marcel Blattner}
\address{Applied AI Research Lab, Lucerne University of Applied Sciences
  and Arts, Suurstoffi 1, 6343 Rotkreuz, Switzerland}
\email{marcel.blattner@hslu.ch}
\date{August 2026}


\begin{document}
\raggedbottom

\begin{abstract}
We identify the graded metabelian locus inside the deformation theory of
positively graded Lie algebras.  Let $M(U)$ be the free metabelian Lie
algebra on $U$ and $B(U)=M(U)'$.  For every finite graded rank vector $h$,
we prove the stack equivalence
\[
\mathcal M^{\mathrm{met,gr}}_{U,h}
\simeq[\operatorname{Quot}^{\mathrm{gr}}_h(B(U))/\GL(U)]
\]
and identify its tangent complex at a quotient $B(U)\twoheadrightarrow C$
with kernel $N$ as
\[
[\mathfrak{gl}(U)\longrightarrow\Hom_{\Sym(U)}(N,C)_0].
\]
For every algebra $\fg$ in this stack, restriction to $\Lambda^2\fg'$ induces
a defect map on $H^2_0(\fg;\fg)$, and we prove that its kernel is
$H^0$ of the Quot-stack tangent complex.  Thus a first-order
deformation is tangent to the metabelian locus exactly when its
derived--derived restriction vanishes.  We also recover the inverse-system
module degree by degree from intrinsic lower-central tensors; on the level
locus its terminal tensor suffices.

For the $14$-dimensional algebra attached to a regular pencil of binary
quartics, exact computation gives $\dim H^2_0=11$ and $H^3_0=0$.  Its
effective miniversal graded slice is smooth of dimension $11$, while its
metabelian locus is a smooth $3$-fold.  Terminal restriction identifies
the $8$-dimensional normal space with derived--derived brackets.  Using the
pencil's classical $V_4$-symmetry, we determine its induced representation
on the graded tangent space and show that four negative-weight primary
obstruction maps are surjective.
\end{abstract}

\maketitle

\section{Introduction}

For a metabelian Lie algebra $\fg$, the adjoint actions of
$U=\fg/\fg'$ commute on $\fg'$:
\[
[\ad_u,\ad_v]|_{\fg'}=\ad_{[u,v]}|_{\fg'}=0,
\]
so $\fg'$ is a $\Sym(U)$-module.  This elementary observation is the link
between metabelian Lie laws and commutative moduli.  The problem addressed
here is to identify that moduli space and its normal directions inside all
graded Lie deformations.

\paragraph{Principal theorem.}
The principal new results are Theorems~\ref{thm:quot-stack} and
\ref{thm:general-defect}.  Fix the degree-one space $U$ and a finite
derived-rank vector $h$, and let $B(U)$ be the derived module of the free
metabelian Lie algebra on $U$.  Formation gives
\[
\mathcal M^{\mathrm{met,gr}}_{U,h}
\simeq[Q_h/\GL(U)],\qquad
Q_h=\operatorname{Quot}^{\mathrm{gr}}_h(B(U)).
\]
At $q:B(U)\twoheadrightarrow C$, with kernel $N$,
the tangent complex is
\[
[\mathfrak{gl}(U)\longrightarrow\Hom_R(N,C)_0],
\]
and restriction of adjoint cocycles gives
\[
\Delta_{\fg}:H^2_0(\fg;\fg)
\longrightarrow\Hom(\Lambda^2\fg',\fg)_0,
\qquad
\ker\Delta_{\fg}
\simeq H^0[\mathfrak{gl}(U)\to\Hom_R(N,C)_0].
\]
Thus the Quot directions are exactly the graded metabelian deformations,
whereas $\Delta_{\fg}$ detects transverse derived--derived brackets.
Corollary~\ref{cor:metabelian-smoothness} records the resulting dimension
formula and an $\Ext^1$-vanishing criterion for formal smoothness.

The novelty here is the relative Lie-theoretic Quot-stack equivalence and
its identification with the derived--derived restriction kernel.  The
Koszul presentation of $B(U)$ and module inverse-system duality are inputs
from \cite{PapadimaSuciu2004,Reutenauer1993,KleimanKleppe2025}.  Tangent
cohomology for $k$-solvable law schemes was developed in
\cite[Appendix]{BarrionuevoTiraoSulca2023}; in the present graded
metabelian setting its linearized identity reduces to
$\phi|_{\Lambda^2\fg'}=0$, and the Quot geometry identifies the kernel.
For the broader Koszul-module context, see
\cite[Proposition~14.5]{Suciu2026}.

\paragraph{Structural consequences.}
Proposition~\ref{prop:terminal-tensor} recovers every homogeneous
inverse-system piece from an intrinsic lower-central tensor.  It follows
that abstract and graded isomorphism have the same orbit criterion for all
formation algebras; on the level locus the terminal tensor alone suffices.
The degree-zero boundary recovers the classical alternating-map problem
for two-step nilpotent Lie algebras
\cite{LegerLuks1972,Gauger1973,GalitskiTimashev1999}.  In rank three the
parameters become determinant-twisted divergence-free polynomial vector
fields, and Theorem~\ref{thm:scalar-nonfaithful} shows that scalar ideals
lose information.  Rank two reduces to binary inverse systems.

\paragraph{Complete application.}
For the $14$-dimensional formation algebra $\fg_\circ$ of a regular binary
quartic pencil, Theorem~\ref{thm:quartic-synthesis} gives a smooth
$11$-dimensional effective graded slice and a smooth $3$-dimensional
metabelian locus.  Terminal restriction identifies its normal space as
\[
H^2_0/T_{\mathrm{met}}
\simeq\Hom(L_2\otimes L_4,L_6)
\oplus\Hom(\Lambda^2L_3,L_6).
\]
Equivalently, the local picture is $11=3+8$: pencil motion supplies the
metabelian directions and terminal derived--derived brackets supply the
normal directions.  The same calculation determines the effective
$V_4$-action and proves surjectivity of four negative-weight primary
obstruction maps.  Appendix~\ref{app:verification} records a compact exact
certificate; the archive contains the matrices, cocycles, and pivot data.

Section~\ref{sec:formation-functor} proves the Quot-stack and tangent
theorems.  Sections~\ref{sec:level-recovery}--\ref{sec:rank-three} develop
intrinsic recovery and the rank-two/rank-three specializations.  The final
two sections give the quartic-pencil calculation and its formal moduli.

Throughout, $\kk$ is a field of characteristic zero.  Divided-power
notation may therefore be translated into normalized ordinary
derivatives.

\section{Main theorem: Quot-stack moduli and metabelian tangent}
\label{sec:formation-functor}

We first isolate the commutative model and then prove the relative
formation and tangent theorems.

\subsection{The commutator module and its dual}
\label{sec:commutator-module}

Let $U$ be an $n$-dimensional vector space, let $R=\Sym(U)$, and let
\[
M(U)=\operatorname{Lie}(U)/\operatorname{Lie}(U)'',
\qquad B(U)=M(U)'.
\]
The ideal $B(U)$ is abelian.  For $u\in U$ and $b\in B(U)$ set
$u\cdot b=[u,b]$.  The operators $\ad_u|_{B(U)}$ commute, so this action
extends to $R$.  We put basic commutators in internal module degree zero;
internal degree $r$ then has Lie degree $r+2$.

Let
\[
\partial_3:R(-1)\otimes\Lambda^3U\longrightarrow
R\otimes\Lambda^2U
\]
be the degree-zero Koszul map
\[
\partial_3(p\otimes u\wedge v\wedge w)
=pu\otimes v\wedge w-pv\otimes u\wedge w
+pw\otimes u\wedge v.
\]

\begin{proposition}[{Koszul presentation; \cite[Theorem~6.2]{PapadimaSuciu2004}}]
\label{prop:koszul-presentation}
There is a natural graded $R$-module presentation
\[
R(-1)\otimes\Lambda^3U\xrightarrow{\partial_3}
R\otimes\Lambda^2U\longrightarrow B(U)\longrightarrow0,
\]
where $1\otimes(u\wedge v)$ maps to $[u,v]$.
\end{proposition}

Let
\[
D=R^\circ=\bigoplus_{r\ge0}R_r^*
\]
be the restricted graded divided-power dual, with contraction action
defined by
\[
\langle b,p\circ\Phi\rangle=\langle pb,\Phi\rangle.
\]
Degreewise dualizing Proposition~\ref{prop:koszul-presentation} gives
\[
B(U)^\circ
=\ker\left(
D\otimes\Lambda^2U^*
\xrightarrow{\delta}
D\otimes\Lambda^3U^*
\right),
\]
where
\[
(\delta\Phi)(u,v,w)
=u\circ\Phi(v,w)-v\circ\Phi(u,w)+w\circ\Phi(u,v).
\]
Write
\begin{equation}\label{eq:closed-two-forms}
\cZ^2(U)=B(U)^\circ.
\end{equation}
This is the closed-form model implicit in the free-Lie and Koszul
descriptions of \cite{Reutenauer1993,Kapranov2012,Kapranov2015}.

\subsection{Formation}

For a graded submodule on either side of the degreewise perfect pairing
$B(U)\times B(U)^\circ\to\kk$, write $(-)^\perp$ for its annihilator.
For finite-codimensional $N\subset B(U)$ and finite-dimensional
$W\subset\cZ^2(U)$, degreewise duality gives
\[
(N^\perp)^\perp=N,\qquad (W^\perp)^\perp=W,\qquad
(B(U)/W^\perp)^\circ\simeq W.
\]
See \cite[Theorem~3.5]{KleimanKleppe2025}.

\begin{definition}[formation functor]
\label{def:formation}
For a finite-dimensional graded submodule $W\subset\cZ^2(U)$ set
\begin{equation}\label{eq:gW}
\mathfrak F(U,W)=\fg_W:=M(U)/W^\perp.
\end{equation}
\end{definition}

Let $\mathsf{InvMet}^{\mathrm{gr}}$ be the groupoid whose objects are pairs
$(U,W)$ as above and whose morphisms
$P:(U,W)\to(U',W')$ are linear isomorphisms such that the induced
restricted-dual pullback carries $W'$ onto $W$.  Let
$\mathsf{Met}^{\mathrm{gr}}$ be the groupoid of finite-dimensional
positively graded metabelian Lie algebras generated by degree one, with
graded isomorphisms.

\begin{theorem}[inverse-system equivalence]
\label{thm:general-equivalence}
The assignment $\mathfrak F:(U,W)\mapsto\fg_W$ is a functor and an
equivalence of groupoids
\[
\mathsf{InvMet}^{\mathrm{gr}}
\simeq
\mathsf{Met}^{\mathrm{gr}}.
\]
In particular, its essential image consists exactly of the
finite-dimensional positively graded metabelian Lie algebras generated by
their degree-one component.
\end{theorem}

\begin{proof}
Naturality of the free Lie algebra and of the Koszul presentation makes
\eqref{eq:gW} functorial.  Conversely, let $\fg$ be generated by
$U=\fg_1$ and be metabelian.  The universal map
$M(U)\twoheadrightarrow\fg$ is the identity in degree one.  Its homogeneous
kernel therefore lies in $B(U)$ and is an $R$-submodule $N$.  Hence
\[
\fg\simeq M(U)/N=\fg_{N^\perp}.
\]
Finite dimensionality makes $N^\perp$ finite dimensional.  This proves
essential surjectivity.

A graded map between two such quotients is determined by its restriction to
degree one.  It descends precisely when its induced map on $B(U)$ carries
one kernel to the other, equivalently when the restricted-dual pullback
carries the corresponding inverse-system submodule onto the other.  This
proves full faithfulness.
\end{proof}

\subsection{Relative inverse systems and the graded Quot stack}
\label{subsec:relative-formation}

Let $A$ be a commutative $\kk$-algebra and put
\[
U_A=A\otimes_\kk U,\qquad
R_A=\Sym_A(U_A),\qquad
B_A=A\otimes_\kk B(U).
\]
Every graded piece of $B_A$ is finite free over $A$.  Its restricted
dual is
\[
B_A^\vee=\bigoplus_{r\ge0}\Hom_A((B_A)_r,A)
\simeq
\ker\left(
D_A\otimes_A\Lambda^2U_A^\vee
\xrightarrow{\delta_A}
D_A\otimes_A\Lambda^3U_A^\vee
\right),
\]
where $D_A=\bigoplus_r\Hom_A((R_A)_r,A)$.  Denote this relative module of
closed divided-power $2$-forms by $\cZ_A^2(U)$.

Fix a finite-support sequence $h=(h_r)_{r\ge0}$ of nonnegative integers.

\begin{proposition}[relative inverse-system duality]
\label{prop:relative-inverse-systems}
Annihilators give mutually inverse, base-change-compatible
correspondences between
\begin{enumerate}[label=\textup{(\roman*)}]
\item homogeneous $R_A$-submodules $N\subset B_A$ for which
$C=B_A/N$ has $C_r$ finite projective of rank $h_r$;
\item finite-support contraction-stable submodules
$W\subset\cZ_A^2(U)$ for which every $W_r$ is an $A$-direct summand of
$(B_A^\vee)_r$ of rank $h_r$.
\end{enumerate}
Under the correspondence,
\[
W=N^\perp\simeq C^\vee,\qquad
C\simeq W^\vee,\qquad N=W^\perp.
\]
\end{proposition}

\begin{proof}
Base change in the Koszul presentation and degreewise dualization give the
displayed description of $B_A^\vee$.  If $C_r$ is finite projective, then
$0\to N_r\to(B_A)_r\to C_r\to0$ splits over $A$, so dualization identifies
$N_r^\perp$ with $C_r^\vee$.  Conversely, a direct-summand inclusion
$W_r\hookrightarrow(B_A^\vee)_r$ dualizes to a surjection
$(B_A)_r\twoheadrightarrow W_r^\vee$ with kernel $W_r^\perp$.  The
split exact sequences remain exact after every map $A\to A'$, so
$(N^\perp)\otimes_AA'=(N\otimes_AA')^\perp$ degree by degree; the same
argument applies to $W^\perp$.  Finally,
$\langle rn,w\rangle=\langle n,r\circ w\rangle$, and the degreewise
double-annihilator identities show that $N$ is $R_A$-stable if and only
if $N^\perp$ is contraction-stable.  They also give the two asserted
double-annihilator identities; compare
\cite[Theorem~3.5]{KleimanKleppe2025}.
\end{proof}

Let
\[
Q_h=\operatorname{Quot}^{\mathrm{gr}}_h(B(U))
\]
be the graded Quot scheme of these quotients.  If
$h=0$, this is a point; for inadmissible $h$ it may be empty.  Otherwise,
with
$s=\max\{r:h_r\ne0\}$, it is the closed incidence subscheme of the product
of quotient Grassmannians
\[
\prod_{r=0}^s\operatorname{Gr}_{h_r}(B(U)_r)
\]
on which the universal kernels satisfy
$U N_r\subseteq N_{r+1}$ for $0\le r<s$.  Thus $Q_h(A)$ consists of
homogeneous quotients $B_A\twoheadrightarrow C_A$ with graded ranks $h$,
up to isomorphism of the target.

Let $\mathcal M^{\mathrm{met,gr}}_{U,h}$ be the following stack over
$\operatorname{Spec}\kk$.  Over a scheme $S$, its objects are positively
graded metabelian $\mathcal O_S$-Lie algebras
$\fg_S=\bigoplus_{i\ge1}(\fg_S)_i$, generated by the locally free bundle
$(\fg_S)_1$ of rank $\dim U$, such that every graded piece is locally free
and
\[
\operatorname{rank}(\fg_S)_{r+2}=h_r\qquad(r\ge0).
\]
Morphisms are graded Lie isomorphisms.  We let $\GL(U)$ act on $Q_h$ by
$a\cdot q=q\circ\widetilde a^{-1}$, where $\widetilde a$ is the induced
automorphism of $B(U)$.

\begin{theorem}[graded Quot-stack description]
\label{thm:quot-stack}
Formation induces an equivalence
\[
\mathcal M^{\mathrm{met,gr}}_{U,h}
\simeq[Q_h/\GL(U)].
\]
For a local $\kk$-algebra $A$, this says that the groupoid of such
$A$-Lie algebras, after choosing a frame of their degree-one part, is the
action groupoid $[Q_h(A)/\GL(U)(A)]$.
\end{theorem}

\begin{proof}
Put $G=\GL(U)$ and let $S$ be a $\kk$-scheme.  For
$\fg_S\in\mathcal M^{\mathrm{met,gr}}_{U,h}(S)$, let
\[
P=\operatorname{Isom}_S(U\otimes\mathcal O_S,(\fg_S)_1)
\longrightarrow S
\]
be its right $G$-torsor of degree-one frames.  On $P$ the tautological
frame gives a surjection from the free metabelian Lie algebra
$M_P(U\otimes\mathcal O_P)$ onto $\fg_P$.  It is the identity in degree
one.  Hence its kernel is a homogeneous
$\Sym_{\mathcal O_P}(U\otimes\mathcal O_P)$-submodule
$N_P\subset B(U)\otimes\mathcal O_P$.  Since $\fg_P$ is generated in
degree one,
\[
B(U)\otimes\mathcal O_P,/N_P\simeq\fg_P'
=\bigoplus_{r\ge0}(\fg_P)_{r+2}.
\]
The pieces on the right are locally free of ranks $h_r$; consequently
each $(N_P)_r$ is a direct summand of $B(U)_r\otimes\mathcal O_P$.
Thus $N_P$ defines a map $P\to Q_h$.  Changing the tautological frame
changes the quotient by the displayed $G$-action, so this map is
$G$-equivariant.

Conversely, an $S$-point of $[Q_h/G]$ is a right $G$-torsor $P\to S$
together with a $G$-equivariant map $P\to Q_h$.  Pulling back the
universal quotient gives a $G$-equivariant exact sequence
\[
0\longrightarrow N_P\longrightarrow B(U)\otimes\mathcal O_P
\longrightarrow C_P\longrightarrow0
\]
whose graded pieces $C_{P,r}$ are locally free of ranks $h_r$.  Formation
produces the $G$-equivariant graded Lie algebra
\[
M_P(U\otimes\mathcal O_P)/N_P.
\]
Its degree-one part is $U\otimes\mathcal O_P$, and its derived algebra is
$C_P$.  The bracket and grading descend fppf-locally along $P\to S$,
giving an object of $\mathcal M^{\mathrm{met,gr}}_{U,h}(S)$ whose
degree-one bundle is $P\times^G U$.

Both constructions commute with base change: the free metabelian and
Koszul constructions do so degreewise, and the kernels above are kernels
of locally split quotients.  They are inverse on the frame torsor, hence
inverse after fppf descent.  The same argument applies to morphisms: a
graded Lie isomorphism is determined by its degree-one restriction, and
after framing it preserves precisely the corresponding kernel in $B(U)$.
This proves the equivalence of stacks.  If $S=\operatorname{Spec}A$ with
$A$ local, the degree-one bundle is free; choosing a frame gives the
stated action groupoid.
\end{proof}

\begin{proposition}[Quot tangent and obstruction spaces]
\label{prop:quot-tangent}
Let $q\in Q_h(\kk)$ correspond to
\[
0\longrightarrow N\longrightarrow B(U)
\longrightarrow C\longrightarrow0.
\]
Then
\[
T_qQ_h\simeq\Hom_R(N,C)_0,
\]
and the first-order tangent complex of $[Q_h/\GL(U)]$ at $q$ is
\begin{equation}\label{eq:quot-tangent-complex}
\left[
\mathfrak{gl}(U)\xrightarrow{\mathrm d\operatorname{orb}_q}
\Hom_R(N,C)_0
\right],
\end{equation}
with $\mathfrak{gl}(U)$ in degree $-1$.  For every small extension
$A'\twoheadrightarrow A$ with square-zero kernel $J$ and every quotient
$q_A$ with kernel $N_A$ and target $C_A$, the standard Quot obstruction
class lies in
\[
\Ext^1_{R_A}(N_A,C_A\otimes_AJ)_0.
\]
It vanishes if and only if $q_A$ lifts to $A'$.
For a small extension of local Artin $\kk$-algebras based at $q$, the
fixed special-fiber obstruction space is
$\Ext^1_R(N,C)_0\otimes_\kk J$.
With the action convention above, if $a\in\mathfrak{gl}(U)$ and
$\widetilde a$ is its induced infinitesimal action on $B(U)$, then
\begin{equation}\label{eq:orbit-derivative}
(\mathrm d\operatorname{orb}_q a)(n)=q(\widetilde a n)
\qquad(n\in N).
\end{equation}
\end{proposition}

\begin{proof}
Let $A_\epsilon=\kk[\epsilon]/(\epsilon^2)$ and choose a degreewise
splitting $B(U)=N\oplus S$ of $q$.  A direct summand
$N_\epsilon\subset B(U)\otimes A_\epsilon$ reducing to $N$ is the graph
of a degree-zero map $N\to S\simeq C$.  It is an
$R\otimes A_\epsilon$-submodule exactly when its class
$\eta:N\to C$ satisfies
$\eta(rn)=r\eta(n)$ for $r\in R$ and $n\in N$.  Changing the splitting
changes a lift of $\eta$, but not its class in $\Hom_R(N,C)_0$.  This
gives $T_qQ_h=\Hom_R(N,C)_0$.

The tangent complex of an action quotient is the infinitesimal action map
into this tangent space.
For $1+\epsilon a\in\GL(U)(\kk[\epsilon]/(\epsilon^2))$, the moved kernel
is $(1+\epsilon\widetilde a)N$; its graph map is
$n\mapsto q(\widetilde a n)$, proving
\eqref{eq:orbit-derivative}.  Standard Quot deformation theory gives the
displayed relative $\Ext^1$ obstruction.  For a small Artin extension,
$J$ is a $\kk$-vector space through the residue map and passage to the
special fiber gives the stated fixed obstruction space.  Quotienting by
the smooth group $\GL(U)$ adds no obstruction.
\end{proof}

For a graded Lie algebra, write $H^q_0$ for the degree-zero part of its
adjoint Chevalley--Eilenberg cohomology.

The tangent cohomology of $k$-solvable law schemes is described in
\cite[Appendix]{BarrionuevoTiraoSulca2023}.  The next theorem identifies
its graded metabelian restriction with the graded Quot-stack tangent.

\begin{theorem}[general metabelian defect and Quot tangent]
\label{thm:general-defect}
Let $\fg=\bigoplus_{i\ge1}L_i$ be a finite-dimensional positively graded
metabelian Lie algebra generated by $L_1$, and put
$\mathfrak d=\fg'=\bigoplus_{i\ge2}L_i$.  Restriction induces a
well-defined map
\[
\Delta_{\fg}:H^2_0(\fg;\fg)
\longrightarrow\Hom(\Lambda^2\mathfrak d,\fg)_0,
\qquad
\Delta_{\fg}[\phi]=\phi|_{\Lambda^2\mathfrak d}.
\]
Its kernel is the tangent space to graded metabelian deformations modulo
infinitesimal graded equivalence:
\begin{equation}\label{eq:general-metabelian-tangent}
\ker\Delta_{\fg}
=\frac{Z^2_0(\fg;\fg)\cap
\ker(\operatorname{res}_{\Lambda^2\mathfrak d})}
{B^2_0(\fg;\fg)}.
\end{equation}
If $\fg$ corresponds to $B(U)\twoheadrightarrow C$ with kernel $N$, then
there is a natural identification
\begin{equation}\label{eq:defect-quot-identification}
\ker\Delta_{\fg}\simeq
H^0\!\left(T_q[Q_h/\GL(U)]\right)\simeq
\operatorname{coker}\left(
\mathfrak{gl}(U)\xrightarrow{\mathrm d\operatorname{orb}}
\Hom_R(N,C)_0
\right).
\end{equation}
Moreover,
$H^{-1}(T_q[Q_h/\GL(U)])\simeq\Der_0(\fg)$.
\end{theorem}

\begin{proof}
Put $B=B(U)$, let $q:B\twoheadrightarrow C$ have kernel $N$, and identify
$\fg=M(U)/N=U\oplus C$.  A degree-zero $2$-cochain sends
$L_i\wedge L_j$ to $L_{i+j}$ and therefore takes values in
$\mathfrak d=\bigoplus_{i\ge2}L_i$.  If
$\mu_\epsilon=\mu+\epsilon\phi$, the coefficient of $\epsilon$ in
$\mu_\epsilon(\mu_\epsilon(x,y),\mu_\epsilon(z,w))$ is
\begin{equation}\label{eq:linearized-metabelian-general}
\phi([x,y],[z,w])+[\phi(x,y),[z,w]]
+[[x,y],\phi(z,w)]=0.
\end{equation}
The last two terms vanish since their entries lie in the abelian ideal
$\mathfrak d$.  As $\mathfrak d$ is spanned by commutators, the
linearized metabelian identity is equivalent to
$\phi|_{\Lambda^2\mathfrak d}=0$; the first-order Jacobi identity is
$d\phi=0$.

If $f\in C^1_0(\fg;\fg)$ and $a,b\in\mathfrak d$, then
$f(a),f(b)\in\mathfrak d$ and
\[
(df)(a,b)=[a,f(b)]-[b,f(a)]-f([a,b])=0.
\]
Thus restriction kills every weight-zero coboundary, proving
\eqref{eq:general-metabelian-tangent}.

We next construct the comparison with the Quot tangent.  Choose a graded
$\kk$-linear splitting $s:C\to B$ of $q$, put
$p_N=1-sq:B\to N$, extend $p_N$ by zero on $U$, and write
\[
\sigma=\operatorname{id}_U\oplus s:\fg=U\oplus C\longrightarrow M(U),
\qquad \pi:M(U)\longrightarrow\fg.
\]
Let $A_\epsilon=\kk[\epsilon]/(\epsilon^2)$.  For
$\eta\in\Hom_R(N,C)_0$, define $N_\eta\subset B\otimes A_\epsilon$ as
the image of
\[
N\otimes A_\epsilon\longrightarrow B\otimes A_\epsilon,
\qquad
n+\epsilon n'\longmapsto n+\epsilon\bigl(n'+s\eta(n)\bigr).
\]
This is degreewise a direct summand.  For $r\in R$ and $n\in N$, the
difference between $r(n+\epsilon s\eta(n))$ and
$rn+\epsilon s\eta(rn)$ lies in $\epsilon N$ exactly when
$r\eta(n)=\eta(rn)$.  Thus $N_\eta$ is an
$R\otimes A_\epsilon$-submodule exactly when $\eta$ is $R$-linear.

There is an $A_\epsilon$-linear isomorphism
\[
\rho_\eta:(M(U)\otimes A_\epsilon)/N_\eta
\xrightarrow{\sim}\fg\otimes A_\epsilon
\]
given by
\[
\rho_\eta(\overline{m+\epsilon m'})
=\pi(m)+\epsilon\bigl(\pi(m')-\eta(p_Nm)\bigr).
\]
For $x,y\in\fg$, set
\[
\kappa_s(x,y)=p_N[\sigma x,\sigma y]_{M(U)}
=[\sigma x,\sigma y]_{M(U)}-\sigma[x,y]_{\fg}\in N.
\]
Transporting the quotient bracket by $\rho_\eta$ gives
\begin{equation}\label{eq:quot-cocycle-map}
[x,y]_\eta=[x,y]_{\fg}+\epsilon\Phi_s(\eta)(x,y),
\qquad
\Phi_s(\eta)(x,y)=-\eta\bigl(\kappa_s(x,y)\bigr).
\end{equation}
Consequently $\Phi_s(\eta)$ is a degree-zero cocycle and vanishes on
$\Lambda^2\mathfrak d$, since $\sigma(\mathfrak d)\subset B$ and $B$ is
abelian.  Replacing $s$ changes $\rho_\eta$ by an infinitesimal graded
change of basis and hence changes $\Phi_s(\eta)$ by a coboundary.  We
obtain a natural map
\[
\Theta:\Hom_R(N,C)_0\longrightarrow\ker\Delta_{\fg},
\qquad \eta\longmapsto[\Phi_s(\eta)].
\]

For $a\in\mathfrak{gl}(U)$, let $\widetilde a$ be the degree-zero
derivation of $M(U)$ induced by $a$, put
$\eta_a=\mathrm d\operatorname{orb}_q(a)$ as in
\eqref{eq:orbit-derivative}, and set
$f_a=\pi\widetilde a\sigma\in C^1_0(\fg;\fg)$.  Since $\widetilde a$ is a
derivation,
\[
(df_a)(x,y)=\pi\widetilde a\bigl(\kappa_s(x,y)\bigr)
=\eta_a\bigl(\kappa_s(x,y)\bigr),
\]
and therefore $\Phi_s(\eta_a)=-df_a$.  Thus $\Theta$ induces
\[
\overline\Theta:\operatorname{coker}\bigl(
\mathfrak{gl}(U)\xrightarrow{\mathrm d\operatorname{orb}_q}
\Hom_R(N,C)_0\bigr)\longrightarrow\ker\Delta_{\fg}.
\]

Conversely, let
$\phi\in Z^2_0(\fg;\fg)$ vanish on
$\Lambda^2\mathfrak d$.  Equation
\eqref{eq:linearized-metabelian-general} shows that
$\mu+\epsilon\phi$ is metabelian.  It remains generated by $U_A$: the
quotient by the graded subalgebra generated by $U_A$ has zero special
fiber and therefore vanishes degreewise by Nakayama.  Its components in
degrees at least two are the fixed free $A$-modules $L_i\otimes A$, so
the derived ranks remain $h$.  With its given degree-one frame,
Theorem~\ref{thm:quot-stack}, with the canonical frame of
$U\otimes A_\epsilon$, produces a quotient deformation of $q$, and the
construction above recovers $[\phi]$.  Hence $\overline\Theta$ is
surjective.  If $\Theta(\eta)=0$, the transported Lie deformation is
infinitesimally isomorphic to the constant one.  By full faithfulness in
Theorem~\ref{thm:quot-stack}, such an isomorphism is determined by its
degree-one component $1+\epsilon a$; after the target isomorphism already
built into the Quot functor, this says precisely that
$\eta=\mathrm d\operatorname{orb}_q(a)$.  Hence $\overline\Theta$ is
injective, proving \eqref{eq:defect-quot-identification}.

Finally,
\[
\ker(\mathrm d\operatorname{orb}_q)
=\{a\in\mathfrak{gl}(U):\widetilde a(N)\subset N\}.
\]
Such an $a$ descends to a degree-zero derivation of $\fg$.  Conversely,
every degree-zero derivation of $\fg$ is determined by its restriction to
$U=\fg_1$; the universal property of $M(U)$ then shows that its
restriction preserves $N$.  This identifies the displayed kernel with
$\Der_0(\fg)$ and proves the final assertion.
\end{proof}

\begin{corollary}[tangent dimension and unobstructedness criterion]
\label{cor:metabelian-smoothness}
In the notation of Theorem~\ref{thm:general-defect},
\[
\dim\ker\Delta_{\fg}
=\dim\Hom_R(N,C)_0-(\dim U)^2+\dim\Der_0(\fg).
\]
If $\Ext^1_R(N,C)_0=0$, then $Q_h$ is smooth at $q$ and the graded
metabelian deformation stack is formally smooth at $[\fg]$; equivalently,
every deformation in $\mathcal M^{\mathrm{met,gr}}_{U,h}$ based at $\fg$
lifts across every small extension of local Artin $\kk$-algebras.
\end{corollary}

\begin{proof}
Taking Euler characteristics in the two-term complex
\eqref{eq:quot-tangent-complex} and using
$H^{-1}\simeq\Der_0(\fg)$ and
$H^0\simeq\ker\Delta_{\fg}$ gives the dimension formula.  If
$\Ext^1_R(N,C)_0=0$, the Quot obstruction theory in
Proposition~\ref{prop:quot-tangent} lifts every deformation based at $q$
across every small extension.  Thus $Q_h$ is formally smooth at $q$;
since it is of finite type over $\kk$, it is smooth there.  The atlas
$Q_h\to[Q_h/\GL(U)]$ is smooth because $\GL(U)$ is smooth, so the quotient
stack is formally smooth at $[\fg]$.
\end{proof}

\section{Intrinsic recovery and level quotients}
\label{sec:level-recovery}

Our second structural result recovers each inverse-system component from
the Lie algebra.  Write $W_r$ for divided-power degree $r$.

\begin{definition}
A finite-dimensional inverse-system module $W$ is \emph{level of degree
$d$} if $W_d\ne0$, $W_r=0$ for $r>d$, and
\[
W=R\circ W_d.
\]
\end{definition}

\begin{lemma}\label{lem:level-socle}
Let $B_W=\fg_W'$.  The module $W=B_W^\circ$ is level of degree $d$ if and
only if the $R$-socle of $B_W$ is concentrated in internal degree $d$.  In
that case
\[
\Soc_R(B_W)=(B_W)_d.
\]
\end{lemma}

\begin{proof}
For $r<d$, absence of socle in degree $r$ is the injectivity of
\[
(B_W)_r\longrightarrow\Hom(U,(B_W)_{r+1}),\qquad
b\longmapsto(u\mapsto ub).
\]
Its dual is the surjectivity of
\[
U\otimes W_{r+1}\longrightarrow W_r.
\]
Iterating gives $W=R\circ W_d$.  The converse is the same argument in
reverse.
\end{proof}

Put $V=\fg_W/\gamma_2(\fg_W)$.  Since $B_W$ is generated in internal
degree zero,
\[
\gamma_j(\fg_W)=\bigoplus_{i\ge j}(\fg_W)_i\qquad(j\ge2).
\]
For $r\ge0$, $p=u_1\cdots u_r\in\Sym^rV$, and $v,w\in V$, choose lifts to
$\fg_W$ and define
\begin{equation}\label{eq:degreewise-tensor}
T_{\fg_W}^{(r)}(p\otimes(v\wedge w))
=\ad_{u_r}\cdots\ad_{u_1}[v,w]
\pmod{\gamma_{r+3}(\fg_W)}.
\end{equation}
For $r=0$ this means the bracket modulo $\gamma_3$.  If
$W\ne0$ and $d=\max\{r:W_r\ne0\}$, then
$\gamma_{d+3}(\fg_W)=0$ and the terminal map is
\begin{equation}\label{eq:terminal-tensor}
T_{\fg_W}:=T_{\fg_W}^{(d)}:
\Sym^dV\otimes\Lambda^2V\longrightarrow\gamma_{d+2}(\fg_W).
\end{equation}

\begin{proposition}[intrinsic degreewise recovery]
\label{prop:terminal-tensor}
Every $T_{\fg_W}^{(r)}$ is well defined and symmetric in its outer entries.
Under $V\simeq U$,
\begin{equation}\label{eq:degreewise-recovery}
\im\bigl((T_{\fg_W}^{(r)})^*\bigr)
=W_r\subset\Sym^rU^*\otimes\Lambda^2U^*.
\end{equation}
Consequently the full tuple of lower-central tensors recovers $W$.  If
$W\ne0$, then in particular
\begin{equation}\label{eq:terminal-recovery}
\im(T_{\fg_W}^*)=W_d.
\end{equation}
\end{proposition}

\begin{proof}
Changing an outer lift by an element of $\gamma_2$ has no effect because
$\fg_W'$ is abelian.  Changing either lift in the initial pair changes its
bracket by an element of $\gamma_3$; after $r$ further brackets the change
lies in $\gamma_{r+3}$.  Moreover
$[\ad_u,\ad_v]=\ad_{[u,v]}$ vanishes on $\fg_W'$, proving symmetry.

The universal iterated-commutator map
\[
\Sym^rU\otimes\Lambda^2U\longrightarrow B(U)_r
\]
passes to
\[
\Sym^rU\otimes\Lambda^2U\longrightarrow
(B_W)_r\simeq\gamma_{r+2}(\fg_W)/\gamma_{r+3}(\fg_W).
\]
Dualizing $B(U)_r\twoheadrightarrow(B_W)_r$ identifies its dual image with
$W_r$.
\end{proof}

\begin{corollary}[orbit classification without levelness]
\label{cor:all-orbits}
For arbitrary finite-dimensional graded submodules
$W\subset\cZ^2(U)$ and $W'\subset\cZ^2(U')$, the following are equivalent:
\begin{enumerate}[label=\textup{(\roman*)}]
\item $\fg_W$ and $\fg_{W'}$ are isomorphic as Lie algebras;
\item they are isomorphic as graded Lie algebras;
\item there is an isomorphism $P:U\to U'$ whose restricted-dual pullback
satisfies $P^*W'=W$.
\end{enumerate}
\end{corollary}

\begin{proof}
An abstract isomorphism preserves the lower-central filtration and induces
an isomorphism of associated graded Lie algebras.  The displayed formula
for $\gamma_j$ identifies each formation algebra with its associated
graded.  Thus \textup{(i)} implies \textup{(ii)}, while
Theorem~\ref{thm:general-equivalence} gives the equivalence of
\textup{(ii)} and \textup{(iii)}.
\end{proof}

\begin{corollary}[level-orbit classification]
\label{cor:level-orbits}
Let $W\subset\cZ^2(U)$ and $W'\subset\cZ^2(U')$ be level of the same
degree $d$.  Then $\fg_W$ and $\fg_{W'}$ are isomorphic, abstractly or
graded, if and only if there is an isomorphism $P:U\to U'$ such that
$P^*W'_d=W_d$.
\end{corollary}

\begin{proof}
Use Corollary~\ref{cor:all-orbits} and
$W=R\circ W_d$, $W'=R\circ W'_d$.
\end{proof}

\begin{remark}[the two-step boundary]\label{rem:two-step-boundary}
For a level module of degree $d=0$, one has
$\cZ^2(U)_0=\Lambda^2U^*$ and $W=W_0$.  Then
\[
\fg_W=U\oplus W_0^*
\]
is two-step nilpotent, with central derived algebra $W_0^*$ and bracket
\[
\Lambda^2U\twoheadrightarrow W_0^*,\qquad
u\wedge v\longmapsto\bigl(\phi\mapsto\phi(u,v)\bigr).
\]
Conversely, every nonabelian Lie algebra of nilpotency class two has this
form after choosing a complement to its derived algebra.  Hence
Corollary~\ref{cor:level-orbits} specializes exactly to the classical
alternating-map orbit problem of
\cite{Gauger1973,GalitskiTimashev1999}; duality on this two-step category
was studied by Leger and Luks \cite{LegerLuks1972}.
\end{remark}

\subsection{A sharp realization bound}

Module catalecticants give the minimum size of a realization with fixed
terminal tensor.

For $0\le r\le d$ define
\begin{equation}\label{eq:module-catalecticant}
C_r^W:B(U)_r\longrightarrow
\Hom(\Sym^{d-r}U,W_d^*)
\end{equation}
by
\[
C_r^W(b)(q)(\Phi)=\langle qb,\Phi\rangle.
\]
Closedness means that the formula factors through the Koszul quotient.
Since $W_r=R_{d-r}\circ W_d$,
\[
\ker C_r^W=(W^\perp)_r,\qquad
\rank C_r^W=\dim W_r=\dim(B_W)_r.
\]

\begin{definition}\label{def:terminal-marking}
A \emph{terminally marked realization} of $W_d$ is a finite-dimensional
graded Lie algebra
\[
\fh=H_1\oplus\cdots\oplus H_{d+2}
\]
such that:
\begin{enumerate}[label=\textup{(\roman*)}]
\item $H_1=U$, and $H_1$ generates $\fh$;
\item $\fh'$ is abelian;
\item an isomorphism $\theta:H_{d+2}\to W_d^*$ is fixed;
\item its terminal tensor agrees with evaluation:
\[
\theta\bigl(T_{\fh}(p\otimes\xi)\bigr)(\Phi)
=\langle p\otimes\xi,\Phi\rangle.
\]
\end{enumerate}
For $0\le r\le d$ let
\[
O_r:H_{r+2}\longrightarrow\Hom(R_{d-r},H_{d+2}),
\qquad O_r(z)(q)=q(\ad)z.
\]
The realization is \emph{terminally observable} if every $O_r$ is
injective.
\end{definition}

\begin{theorem}[module-catalecticant minimality]
\label{thm:minimality}
Every terminally marked realization satisfies
\[
\dim\fh\ge \dim U+\sum_{r=0}^d\rank C_r^W.
\]
Equality holds if and only if it is terminally observable.  At equality
there is a graded isomorphism $\fh\simeq\fg_W$ respecting the terminal
marking.
\end{theorem}

\begin{proof}
Generation and metabelianity give surjective iterated-bracket maps
\[
\mu_r:B(U)_r\twoheadrightarrow H_{r+2}.
\]
The terminal marking gives the factorization
\[
C_r^W=
\Hom(R_{d-r},\theta)\circ O_r\circ\mu_r.
\]
Thus $\rank C_r^W\le\dim H_{r+2}$; summing proves the bound.  Equality in
every degree is equivalent to injectivity of every $O_r$.  Then
\[
\ker\mu_r=\ker C_r^W=(W^\perp)_r
\]
for every $r$, and the universal map $M(U)\to\fh$ induces the claimed
isomorphism.
\end{proof}

\begin{remark}
This is a minimum only in the displayed marked realization class.  In rank
greater than two, levelness also need not exclude central elements in degree
one.  If one wants $Z(\fg_W)$ to equal the terminal layer, one must impose
the common-radical condition
\[
\{u\in U:u\wedge U\subset(W_0)^\perp\}=0.
\]
\end{remark}

\section{Binary level algebras}
\label{sec:binary-equivalence}

When $\dim U=2$, the commutator module is cyclic, so formation reduces to
ordinary binary inverse systems.  The determinant factor below makes this
reduction functorial.

\subsection{Formation equivalence}

Assume $\dim U=2$ and put $\Delta_U=\Lambda^2U$.  Since
$\Lambda^3U=0$,
\[
B(U)=R\otimes\Delta_U,\qquad
B(U)^\circ=D\otimes\Delta_U^*.
\]
For a nonzero subspace $F\subset D_d=\Sym^d(U^*)$, put
\[
A_F=R/\Ann(F).
\]
The inverse-system module generated by $F\otimes\Delta_U^*$ gives the
canonical determinant-twisted Lie algebra
\begin{equation}\label{eq:binary-canonical}
\cL(U,F)=U\oplus(A_F\otimes\Delta_U)
\end{equation}
with
\begin{equation}\label{eq:binary-bracket}
\begin{aligned}
[u,v]&=1\otimes(u\wedge v),\\
[u,a\otimes\eta]&=(ua)\otimes\eta,\\
[A_F\otimes\Delta_U,A_F\otimes\Delta_U]&=0.
\end{aligned}
\end{equation}
Choosing an area form only trivializes $\Delta_U$; it is not part of the
intrinsic input.

Fix $d\ge1$ and $1\le m\le d+1$.  Let $\mathcal P_{d,m}$ be the groupoid
of pairs $(U,F)$, where $\dim U=2$ and $F$ is an $m$-plane in
$\Sym^dU^*$.  A morphism $P:(U,F)\to(U',F')$ is a linear isomorphism
satisfying $P^*F'=F$.

Let $\mathcal M_{d,m}^{\mathrm{lev}}$ be the groupoid of positively graded
metabelian Lie algebras
\[
\fg=L_1\oplus\cdots\oplus L_{d+2}
\]
such that $L_1$ generates $\fg$,
\[
\dim L_1=2,\quad \dim L_2=1,\quad
\operatorname{class}(\fg)=d+2,\quad
Z(\fg)=L_{d+2},\quad \dim L_{d+2}=m.
\]

\begin{theorem}[binary formation equivalence]
\label{thm:binary-equivalence}
The construction \eqref{eq:binary-canonical} is a functor and defines an
equivalence
\[
\mathcal P_{d,m}\simeq\mathcal M_{d,m}^{\mathrm{lev}}.
\]
For a morphism $P$, the corresponding Lie map is
\[
P\oplus(\overline P\otimes\Lambda^2P),
\]
where $\overline P:A_F\to A_{F'}$ is induced by $P$.
\end{theorem}

\begin{proof}
For a target algebra put $U=L_1$ and $R=\Sym(U)$.  Since
$B(U)=R\otimes\Lambda^2U$, Theorem~\ref{thm:general-equivalence}
identifies it with
$U\oplus((R/I)\otimes\Lambda^2U)$ for a unique homogeneous
finite-colength ideal $I$.  No nonzero element of $U$ is central, so
\[
Z(\fg)=\Soc(R/I)\otimes\Lambda^2U.
\]
The target conditions are therefore equivalent to $R/I$ being Artin level
of socle degree $d$ and type $m$.  Macaulay duality gives
$I=\Ann(F)$ for $F=I_d^\perp$
\cite[Lemma~3.2]{Iarrobino2004}.  Functoriality gives the displayed map,
and Proposition~\ref{prop:terminal-tensor} gives full faithfulness.
\end{proof}

\begin{corollary}\label{cor:binary-classification}
Graded isomorphism classes in $\mathcal M_{d,m}^{\mathrm{lev}}$ are the
$\GL(U)$-orbits on $\operatorname{Gr}(m,\Sym^dU^*)$, equivalently the
$\PGL(U)$-orbits.
\end{corollary}

\begin{corollary}[strong Lefschetz adjoint action]
\label{cor:strong-lefschetz-adjoint}
There is a nonempty Zariski-open subset $U_{\mathrm{Lef}}\subset U$ such
that, for every $u\in U_{\mathrm{Lef}}$, $0\le r\le d$, and
$0\le q\le d-r$,
\[
(\ad_u)^q:L_{r+2}\longrightarrow L_{r+q+2}
\]
has maximal rank:
\[
\rank(\ad_u)^q
=\min\{\dim L_{r+2},\dim L_{r+q+2}\}.
\]
\end{corollary}

\begin{proof}
Every Artinian quotient of a polynomial ring in two variables over a
characteristic-zero field has the strong Lefschetz property
\cite[Proposition~4.4]{HarimaMiglioreNagelWatanabe2003}.  Since
$L_{r+2}=(A_F)_r\otimes\Lambda^2U$ and $\ad_u$ acts on the derived algebra
as multiplication by the residue class $\bar u$, this is exactly the
strong Lefschetz maximal-rank statement.  Strong Lefschetz elements form a
nonempty open subset of $(A_F)_1$; take its inverse image under
$U\to(A_F)_1$.
\end{proof}

\subsection{Automorphisms}
\label{sec:automorphisms}

Write $\fg_F=\cL(U,F)$.  Since the grading is generated by $L_1=U$, a
graded automorphism is determined by its action on $U$.

\begin{theorem}[graded automorphisms]\label{thm:graded-aut}
Restriction to degree one induces an isomorphism of algebraic groups
\[
\Aut_{\mathrm{gr}}(\fg_F)
\simeq
\operatorname{Stab}_{\GL(U)}(F).
\]
Consequently,
\[
1\longrightarrow\mathbb G_m
\longrightarrow\Aut_{\mathrm{gr}}(\fg_F)
\longrightarrow\operatorname{Stab}_{\PGL(U)}(F)
\longrightarrow1,
\]
where $\mathbb G_m$ is the Carnot-dilation subgroup.
\end{theorem}

\begin{proof}
If $P\in\GL(U)$ stabilizes $F$, Theorem~\ref{thm:binary-equivalence}
supplies the graded automorphism
$P\oplus(\overline P\otimes\Lambda^2P)$.  Conversely, naturality of the
terminal tensor forces the degree-one restriction of a graded automorphism
to stabilize $F$.  Injectivity follows because $U$ generates $\fg_F$.
A scalar $\lambda\operatorname{id}_U$ acts on $L_j$ by $\lambda^j$.
\end{proof}

Choose a basis $X,Y$ of $U$ and trivialize $\Lambda^2U$.  Write
$A=A_F$, let $x,y\in A_1$ be the coordinate classes, and normalize
\[
[X,Y]=1,\qquad [X,q]=xq,\qquad [Y,q]=yq
\quad(q\in A).
\]
Let $\operatorname{IAut}(\fg_F)$ be the subgroup acting trivially on
$\fg_F/\fg_F'$.  For $a,b\in A$, put
\[
c_{a,b}=1+xb-ya.
\]
It is a unit, since $xb-ya$ lies in the nilpotent graded maximal ideal.
Define
\begin{equation}\label{eq:IA-formula}
\tau_{a,b}(X)=X+a,\qquad
\tau_{a,b}(Y)=Y+b,\qquad
\tau_{a,b}(q)=q\,c_{a,b}.
\end{equation}

\begin{theorem}[full automorphism correspondence]
\label{thm:full-aut}
Every $\tau_{a,b}$ is an automorphism, and
\[
A^2\longrightarrow\operatorname{IAut}(\fg_F),\qquad
(a,b)\longmapsto\tau_{a,b},
\]
is an isomorphism of algebraic varieties.  Under this parametrization the
group law is
\[
(a,b)\star(a',b')
=\bigl(a+c_{a,b}a',\,b+c_{a,b}b'\bigr).
\]
In particular, it is not generally the additive group law on $A^2$.

There is a split exact sequence
\[
1\longrightarrow\operatorname{IAut}(\fg_F)
\longrightarrow\Aut(\fg_F)
\longrightarrow\operatorname{Stab}_{\GL(U)}(F)
\longrightarrow1.
\]
Thus abstract and graded isomorphism have the same orbit criterion for the
algebras $\fg_F$.  The ungraded groupoid also contains the IA
automorphisms.
\end{theorem}

\begin{proof}
One has
\[
[X+a,Y+b]=1+xb-ya=c_{a,b}=\tau_{a,b}(1).
\]
Because $A$ is abelian, the restrictions of $\ad_{X+a}$ and
$\ad_{Y+b}$ to $A$ remain multiplication by $x$ and $y$.  Hence
\[
\tau_{a,b}(p(x,y)1)=p(x,y)c_{a,b}.
\]
This proves bracket preservation; multiplication by the unit $c_{a,b}$ is
invertible.

Conversely, an IA automorphism sends $X$ and $Y$ to $X+a$ and $Y+b$ for
unique $a,b\in A$.  Since $X,Y$ generate, the preceding calculation forces
its restriction to $A$.  Direct composition gives the displayed group law
and $c_{(a,b)\star(a',b')}=c_{a,b}c_{a',b'}$.

Every automorphism preserves the lower-central filtration and induces some
$P\in\GL(U)$ on the abelianization.  The terminal tensor implies
$P^*F=F$.  Composing with the inverse graded lift of $P$ leaves an IA
automorphism.  The graded lift supplies the splitting.
\end{proof}

\begin{remark}[related automorphism formulas]\label{rem:auto-priority}
Related formulas for two-generator metabelian nilpotent Lie algebras appear
in \cite{Papistas1992,Papistas1996,DrenskyFindik2012}.  We include the
parametrization to fix the conventions used below.
\end{remark}

\section{Rank three: vector fields and failure of scalar parameters}
\label{sec:rank-three}

Extend $\delta$ to $D\otimes\Lambda^\bullet U^*$.

\begin{proposition}[potentials; rank-three vector fields]
\label{prop:potentials-divergence}
For every $r\ge0$,
\[
\frac{D_{r+1}\otimes U^*}{\delta(D_{r+2})}
\xrightarrow{\ \sim\ }\cZ^2(U)_r,
\qquad [\alpha]\longmapsto\delta\alpha.
\]
If $\dim U=3$, the canonical map
$U\otimes\Lambda^3U^*\simeq\Lambda^2U^*$ identifies $\cZ^2(U)$ with
\[
\ker\left(
\operatorname{div}:D\otimes U\otimes\Lambda^3U^*
\longrightarrow D\otimes\Lambda^3U^*
\right),
\quad
\operatorname{div}(f\otimes u\otimes\omega)
=(u\circ f)\otimes\omega.
\]
Thus rank-three inverse systems are determinant-twisted divergence-free
polynomial vector fields, with $R$ acting by constant-coefficient
differentiation.
\end{proposition}

\begin{proof}
The Euler homotopy identity
$\delta\iota_E+\iota_E\delta=(r+q)\operatorname{id}$ on
$D_r\otimes\Lambda^qU^*$ gives exactness in positive degree.  Under
$U\otimes\Lambda^3U^*\simeq\Lambda^2U^*$, $\delta$ on $2$-forms is the
displayed divergence.
\end{proof}

Theorem~\ref{thm:general-equivalence} therefore parametrizes rank-three
formation algebras faithfully by $\GL(U)$-orbits of finite-dimensional
graded modules of such vector fields.  Scalar ideals, however, are not
faithful because $B(U)$ is not cyclic.

For a finite-colength homogeneous ideal $I\subset R=\Sym(U)$, the quotient
\[
\fg_I=M(U)/I B(U)
\]
is a finite-dimensional metabelian Lie algebra, with
\[
\fg_I'\simeq
\operatorname{coker}\left(
(R/I)(-1)\otimes\Lambda^3U
\xrightarrow{\partial_3^{R/I}}
(R/I)\otimes\Lambda^2U
\right).
\]
Unlike the binary case, $B(U)$ is not cyclic when $\dim U>2$, and $I$ need
not be recoverable from this quotient module.

\begin{theorem}[failure of scalar faithfulness in rank three]
\label{thm:scalar-nonfaithful}
Let $U=\langle x,y,z\rangle$ and
\[
I=(x^{a+1},y^{b+1},z^{c+1}),\qquad a,b,c\ge1.
\]
Put $s=x^ay^bz^c$ and $d=a+b+c$.  Then
\[
\bigl(B(U)/I B(U)\bigr)_d=0,
\qquad (\fg_I')_{d+2}=0,
\qquad s\,\fg_I'=0,
\]
and consequently
\[
I B(U)=(I+(s))B(U).
\]
The two distinct ideals $I$ and $I+(s)$ therefore define the same
metabelian Lie quotient.
\end{theorem}

\begin{proof}
Let $A=R/I$, and use the same symbols for monomials in $R$ and their
residue classes in $A$.  In degree $d$, apply the Koszul map over $A$ to
\[
x^{a-1}y^bz^c\otimes x\wedge y\wedge z,\quad
x^ay^{b-1}z^c\otimes x\wedge y\wedge z,\quad
x^ay^bz^{c-1}\otimes x\wedge y\wedge z.
\]
Because $x^{a+1}=y^{b+1}=z^{c+1}=0$ in $A$, two terms in each Koszul
image vanish, leaving
\[
s\otimes(y\wedge z),\qquad
-s\otimes(x\wedge z),\qquad
s\otimes(x\wedge y).
\]
Since $A_d=\kk s$, these elements span
$A_d\otimes\Lambda^2U$, so
$\bigl(B(U)/I B(U)\bigr)_d=0$, equivalently
$(\fg_I')_{d+2}=0$.  The same three relations show that multiplication by
$s$ kills every minimal commutator
generator.  Since those generators generate $\fg_I'$ over $A$, one has
$s\fg_I'=0$.  Equivalently, $(I+(s))B(U)=IB(U)$.  Since
$s\notin I$, the ideals are distinct.
\end{proof}

\begin{remark}
Theorem~\ref{thm:scalar-nonfaithful} does not contradict the formation
equivalence of Theorem~\ref{thm:general-equivalence}.  It shows that the
restricted scalar assignment $I\mapsto I B(U)$ forgets information.  The
correct higher-rank parameter is the submodule
$W\subset\cZ^2(U)$ of closed $2$-forms, not a scalar inverse system.
There is not even a nontrivial cyclic ansatz on $U\oplus A$ with
$[u,v]=\beta(u,v)1_A$ and $[u,a]=\bar ua$ when $\dim U>2$: Jacobi would
force
$\beta(v,w)u+\beta(w,u)v+\beta(u,v)w=0$ in $A_1\simeq U$, hence
$\beta=0$.
\end{remark}

\section{The regular quartic pencil: cohomology and pencil motion}
\label{sec:quartic-pencil}

\subsection{The pencil and its metabelian Quot germ}

We now compute how the binary formation locus sits inside all graded
deformations of one regular quartic pencil.

Choose dual coordinates $S,T$ and differential coordinates $x,y$.  For a
pencil $F$, its joint third catalecticant is
\[
C^{(3)}_F:R_3\longrightarrow\Hom(F,U^*),\qquad
p\longmapsto(f\mapsto p\circ f).
\]
In divided-power normalization consider
\begin{equation}\label{eq:principal-pencil}
F_\circ=\operatorname{span}
\{S^4+T^4,\ 4S^3T+4ST^3\}.
\end{equation}
Its joint third catalecticant is invertible, and direct contraction gives
\begin{equation}\label{eq:principal-ann}
\Ann(F_\circ)
=(x^4-y^4,\ x^3y-xy^3,\ x^2y^2).
\end{equation}
Thus
\[
h(A_{F_\circ})=(1,2,3,4,2).
\]
The Lie algebra $\fg_\circ=\fg_{F_\circ}$ is $14$-dimensional of class six,
with Carnot-layer dimensions
\[
(2,1,2,3,4,2).
\]

Use the homogeneous basis
\[
\begin{split}
&X,Y;\ Z_0;\ Z_x,Z_y;\ Z_{x^2},Z_{xy},Z_{y^2};\\
&Z_{x^3},Z_{x^2y},Z_{xy^2},Z_{y^3};\ T_0,T_1.
\end{split}
\]
The bracket is generated by
\[
[X,Y]=Z_0,\qquad
[X,Z_p]=Z_{xp},\qquad
[Y,Z_p]=Z_{yp},
\]
subject to \eqref{eq:principal-ann}.  Equivalently,
\[
\ad_{sX+tY}^4[X,Y]
=(s^4+t^4)T_0+4st(s^2+t^2)T_1.
\]

The main conclusion is stated first.  Its cohomological and pencil-motion
ingredients are proved in the subsections below;
Section~\ref{sec:local-quartic-geometry} completes the local analysis.

\begin{theorem}[quartic-pencil graded deformation theorem]
\label{thm:quartic-synthesis}
Let $S$ be an effective miniversal graded slice at $\fg_\circ$.  Then
\[
S\simeq\operatorname{Spf}\kk[[t_1,\ldots,t_{11}]],
\qquad K:=\operatorname{Stab}_{\mathrm{eff}}(\fg_\circ)\simeq V_4.
\]
There is a $K$-equivariant formally smooth map
\[
\Theta:S\longrightarrow
\widehat{E}_{\mathrm{top},0},\qquad
E_{\mathrm{top}}
=\Hom(L_2\otimes L_4,L_6)\oplus\Hom(\Lambda^2L_3,L_6),
\]
of relative dimension three.  Its zero fiber is the scheme-theoretic
metabelian locus and is the formal binary-pencil germ
$\operatorname{Spf}\kk[[p_1,p_2,p_3]]$.  Moreover,
\[
\mathrm d\Theta:H^2_0(\fg_\circ;\fg_\circ)\twoheadrightarrow
E_{\mathrm{top}}
\]
has kernel the pencil-motion tangent $P$.  Thus the effective graded
tangent decomposes into three metabelian and eight transverse directions.
\end{theorem}

\begin{proof}
Combine Theorem~\ref{thm:cohomology-table},
Corollary~\ref{cor:metabelian-tangent},
Proposition~\ref{prop:terminal-defect}, and
Theorems~\ref{thm:exact-stabilizer}, \ref{thm:formal-graded-hull}, and
\ref{thm:terminal-map}.
\end{proof}

Let
$\mathcal P_{\mathrm{reg}}\subset\operatorname{Gr}(2,\Sym^4U^*)$
be the open locus of pencils with invertible joint third catalecticant.
For every $F\in\mathcal P_{\mathrm{reg}}$, the associated Lie algebra has
the same graded dimensions.  On the finite-projective-stabilizer locus,
the $\PGL_2$-orbits have dimension three.  Hence a Rosenlicht quotient of
a nonempty invariant open subset has dimension
\[
\dim\operatorname{Gr}(2,5)-\dim\PGL_2=6-3=3.
\]
Binary-quartic pencils are classified in \cite{Wall1998}.  For
quotient-orbit methods in the two-generator metabelian case, see
\cite{GavioliMontiYoung2001}.

\begin{proposition}[the metabelian Quot germ]
\label{prop:quartic-quot-germ}
Let $h=(1,2,3,4,2)$ be the internal rank vector of
$\fg_\circ'$, and let $q_\circ\in Q_h$ be its quotient point.  Formation
induces an isomorphism of formal schemes
\[
\widehat{Q_h}_{\,q_\circ}\simeq
\widehat{\mathcal P_{\mathrm{reg}}}_{\,F_\circ}.
\]
Consequently the fixed-$U$ Quot germ is smooth of dimension six, while
$H^0(T_{q_\circ}[Q_h/\GL(U)])$ has dimension three.
\end{proposition}

\begin{proof}
Since $\dim U=2$,
\[
B(U)_r=R_r\otimes\Lambda^2U,\qquad \dim B(U)_r=r+1.
\]
Thus $h_r=\dim B(U)_r$ for $0\le r\le3$.  In the incidence description
of $Q_h$, the universal kernels in these degrees vanish.  Since the
quotient is zero above degree four, the only nontrivial datum is a
rank-two quotient of $B(U)_4$, and all incidence conditions are automatic:
\[
Q_h\simeq\operatorname{Gr}^{\mathrm{quot}}_2(B(U)_4).
\]
Dualizing and using
$B(U)_4^\vee=D_4\otimes\Lambda^2U^*$ identifies this Grassmannian with
$\operatorname{Gr}(2,\Sym^4U^*)$; the top inverse-system piece
corresponding to $F$ is
\[
W_4=F\otimes\Lambda^2U^*.
\]

On $\mathcal P_{\mathrm{reg}}$, the contraction
$U\otimes F\to D_3$ is an isomorphism, since its dual is the joint third
catalecticant.  Hence
$R\circ(F\otimes\Lambda^2U^*)$ has rank vector $h$, and formation agrees
with the Grassmannian identification.  Since
$\mathcal P_{\mathrm{reg}}$ is open and contains $F_\circ$, it induces the
displayed isomorphism of completions.  The dimension assertions follow
from $\dim\operatorname{Gr}(2,5)=6$ and the finite projective stabilizer.
\end{proof}

\begin{corollary}\label{cor:aut-dimension}
For $\fg_\circ$,
\[
\dim\operatorname{IAut}(\fg_\circ)=24,\qquad
\dim\Aut(\fg_\circ)=25.
\]
Its inner automorphism group has dimension $14-\dim Z(\fg_\circ)=12$.
\end{corollary}

\begin{proof}
Here $\dim A_{F_\circ}=12$.  Apply
Theorems~\ref{thm:graded-aut}, \ref{thm:full-aut}, and
\ref{thm:exact-stabilizer}.
\end{proof}

\subsection{Adjoint cohomology}
\label{sec:adjoint-cohomology}

Let $L=\bigoplus_{j=1}^6L_j=\fg_\circ$ and
$C^q(L;L)=\Hom(\Lambda^qL,L)$.  Write $C^q_w$ for the cochains satisfying
\[
f(L_{i_1},\ldots,L_{i_q})
\subseteq L_{i_1+\cdots+i_q+w}.
\]
The differential preserves weight; write
$Z^q_w=\ker(d:C^q_w\to C^{q+1}_w)$,
$B^q_w=\im(d:C^{q-1}_w\to C^q_w)$, and
$H^q_w=Z^q_w/B^q_w$.

We use the Nijenhuis--Richardson convention
\cite{NijenhuisRichardson1966,NijenhuisRichardson1967} for which
\[
[\mu,f]_{\NR}=-df\qquad(f\in C^2(L;L)),
\]
so the Maurer--Cartan equation is
\begin{equation}\label{eq:MC-convention}
-d\alpha+\frac12[\alpha,\alpha]_{\NR}=0.
\end{equation}

The displayed model has integral structure constants.  Flat base change
gives
\begin{equation}\label{eq:CE-base-change}
C^\bullet_w(\fg_{\circ,\mathbb Q};\fg_{\circ,\mathbb Q})
\otimes_{\mathbb Q}\kk
\simeq C^\bullet_w(\fg_\circ;\fg_\circ),\qquad
H^q_w(\fg_{\circ,\mathbb Q};\fg_{\circ,\mathbb Q})
\otimes_{\mathbb Q}\kk
\simeq H^q_w(\fg_\circ;\fg_\circ).
\end{equation}
We therefore compute over $\mathbb Q$.

\begin{theorem}[adjoint cohomology of $\fg_\circ$]
\label{thm:cohomology-table}
For $\fg_\circ$,
\[
\dim H^2_w=
\begin{cases}
19,&w=-2,\\
26,&w=-1,\\
11,&w=0,\\
4,&w=1,\\
0,&\text{otherwise},
\end{cases}
\qquad \dim H^2=60.
\]
Moreover,
\[
\dim H^3_w=
\begin{cases}
34,&w=-7,\\
60,&w=-6,\\
62,&w=-5,\\
39,&w=-4,\\
12,&w=-3,\\
2,&w=-2,\\
0,&\text{otherwise},
\end{cases}
\qquad \dim H^3=209.
\]
\end{theorem}

\begin{proof}
In each weight,
\[
\dim H^2_w=\dim C^2_w-\rank d_1-\rank d_2,\qquad
\dim H^3_w=\dim C^3_w-\rank d_2-\rank d_3.
\]
All ranks are obtained over $\mathbb Q$ by fraction-preserving sparse
Gaussian elimination.  Table~\ref{tab:ce-certificate} records the blocks
that contribute nonzero cohomology; the source archive records every
block, pivot set, and check that $d_{q+1}d_q=0$.
\end{proof}

The weight-zero block has
\[
(\dim C^1_0,\dim C^2_0,\dim C^3_0)=(38,73,25)
\]
and
\[
(\rank d_1,\rank d_2,\rank d_3)=(37,25,0).
\]
Therefore
\begin{equation}\label{eq:H20H30}
\dim H^2_0=11,\qquad H^3_0=0.
\end{equation}

\begin{proposition}[generic weight-zero cohomology]
\label{prop:generic-H20}
On a nonempty Zariski-open subset of the regular
finite-projective-stabilizer pencil locus,
\[
\dim H^2_0=11,\qquad H^3_0=0.
\]
\end{proposition}

\begin{proof}
A finite projective stabilizer gives $\dim\Der_0(\fg_F)=1$, hence
$\rank d_1=37$.  At $F_\circ$, $d_2$ is surjective; openness of maximal
matrix rank proves the claim.
\end{proof}

\subsection{\texorpdfstring{The pencil-motion slice in $H^2_0$}
{The pencil-motion slice in weight-zero cohomology}}
\label{sec:pencil-motion}

We now compute the Kodaira--Spencer image of the quartic-pencil family in
the $11$-dimensional space $H^2_0$.

Write a quartic pencil in divided-power coordinates as the row span of
\[
Q=(q_{ja})\in M_{2\times5}(\kk),
\qquad 0\le j\le1,\quad 0\le a\le4,
\]
where $q_{ja}$ is the coefficient of
$\binom{4}{a}S^{4-a}T^a$ in row $j$.  On the fixed regular
$14$-dimensional graded vector space, only multiplication into the top
layer varies:
\[
x^{4-a}y^a\longmapsto q_{0a}T_0+q_{1a}T_1.
\]
For a coefficient tangent $\dot Q=(\dot q_{ja})$, define a weight-zero
$2$-cochain by
\begin{equation}\label{eq:coefficient-cocycle}
\begin{aligned}
\phi_{\dot Q}(X,Z_{x^{3-i}y^i})
  &=\sum_{j=0}^1\dot q_{j,i}T_j,\\
\phi_{\dot Q}(Y,Z_{x^{3-i}y^i})
  &=\sum_{j=0}^1\dot q_{j,i+1}T_j,
\qquad 0\le i\le3,
\end{aligned}
\end{equation}
and set all its other values to zero.

\begin{proposition}[pencil-motion tangent space]
\label{prop:pencil-slice}
Every $\phi_{\dot Q}$ is a cocycle.  At $F_\circ$, the span of their
classes in $H^2_0(\fg_\circ;\fg_\circ)$ has dimension exactly three.
It is the image of
\[
T_{F_\circ}\operatorname{Gr}(2,\Sym^4U^*)/
T_{F_\circ}(\PGL(U)\cdot F_\circ)
\longrightarrow H^2_0(\fg_\circ;\fg_\circ).
\]
\end{proposition}

\begin{proof}
The two rules in \eqref{eq:coefficient-cocycle} are the linearization of
the commuting $x$- and $y$-actions into the central top layer.  Hence the
linearized Jacobi identity is satisfied.  Equivalently, direct application
of the Chevalley--Eilenberg differential gives
$d\phi_{\dot Q}=0$ for all ten coefficient directions.

Infinitesimal row operations do not move the point of the Grassmannian, and
infinitesimal $\PGL(U)$-motions are induced by changes of generators, hence
give coboundaries.  Conversely,
Theorem~\ref{thm:general-defect} and
Proposition~\ref{prop:quartic-quot-germ} identify this quotient tangent
with the metabelian subspace of $H^2_0$.  Its dimension is
$6-3=3$ by Theorem~\ref{thm:exact-stabilizer}.
\end{proof}

At $F_\circ$, define cochains $\psi_0,\psi_1,\psi_2$ by the
following nonzero values:
\begin{equation}\label{eq:psi-basis}
\begin{aligned}
\psi_0(X,Z_{x^3})&=T_0,\\
\psi_1(X,Z_{x^2y})&=T_0,
&\psi_1(Y,Z_{x^3})&=T_0,\\
\psi_2(X,Z_{xy^2})&=T_0,
&\psi_2(Y,Z_{x^2y})&=T_0.
\end{aligned}
\end{equation}
They correspond respectively to varying the first row by
$S^4$, $4S^3T$, and $6S^2T^2$.

\begin{theorem}[affine integrability of pencil motion]
\label{thm:affine-slice}
The classes $[\psi_0],[\psi_1],[\psi_2]$ form a basis of the pencil-motion
image, and
\[
[\psi_i,\psi_j]_{\NR}=0
\]
as $3$-cochains for every $i,j$.  Consequently,
\[
\mu_{\mathbf t}
=\mu+\sum_{i=0}^2t_i\psi_i
\]
is a Lie bracket for every $\mathbf t\in\kk^3$.
\end{theorem}

\begin{proof}
Exact elimination gives
\[
\dim\bigl(B^2_0+\langle\psi_0,\psi_1,\psi_2\rangle\bigr)
-\dim B^2_0=3.
\]
Every insertion $\psi_i\circ\psi_j$ vanishes: the inner value lies in
$L_6$, whereas each $\psi_i$ is zero whenever one input lies in $L_6$.
The Maurer--Cartan equation \eqref{eq:MC-convention} is therefore
satisfied.
\end{proof}

\begin{corollary}\label{cor:codim-eight}
The pencil-motion subspace $P$ accounts for three of the eleven weight-zero
infinitesimal deformations and has codimension eight in $H^2_0$.
\end{corollary}

\section{Local graded moduli, symmetry, and obstructions}
\label{sec:local-quartic-geometry}

\subsection{Metabelian tangent and terminal defect}
\label{sec:metabelian-tangent}

Let $\mathfrak d=\fg_\circ'$ and let
\[
P=\operatorname{span}\{[\psi_0],[\psi_1],[\psi_2]\}
\subset H^2_0(\fg_\circ;\fg_\circ)
\]
be the pencil-motion subspace.  The following specializes the general
defect theorem; compare the linearized solvable-law schemes of
\cite{BarrionuevoTiraoSulca2023}.

\begin{corollary}[quartic-pencil metabelian tangent]
\label{cor:metabelian-tangent}
Write $\Delta=\Delta_{\fg_\circ}$ for the general defect map
\[
\Delta:H^2_0(\fg_\circ;\fg_\circ)
\longrightarrow\Hom(\Lambda^2\mathfrak d,\fg_\circ)_0.
\]
Then
\[
\ker\Delta=P,\qquad\rank\Delta=8.
\]
Thus a weight-zero first-order deformation is metabelian if and only if
its class belongs to the pencil-motion subspace.
\end{corollary}

\begin{proof}
Theorem~\ref{thm:general-defect} identifies $\ker\Delta$ with $H^0$ of the
graded Quot-stack tangent complex.  By
Proposition~\ref{prop:quartic-quot-germ}, this is the tangent to
$\operatorname{Gr}(2,\Sym^4U^*)$ modulo changes of generators.
Proposition~\ref{prop:pencil-slice} identifies that three-dimensional
space with $P$.  Since $\dim H^2_0=11$, the defect has rank eight.
\end{proof}

The defect is already detected in the terminal layer.  Put
\[
E_{\mathrm{top}}
=\Hom(L_2\otimes L_4,L_6)
\oplus\Hom(\Lambda^2L_3,L_6),
\qquad \dim E_{\mathrm{top}}=6+2=8,
\]
and let $\Delta_{\mathrm{top}}$ retain only these two components of
$\Delta$.

\begin{proposition}[terminal defect quotient]
\label{prop:terminal-defect}
Relative to the fixed grading, terminal restriction induces a canonical
isomorphism
\[
H^2_0(\fg_\circ;\fg_\circ)/P
\xrightarrow{\ \sim\ }E_{\mathrm{top}}.
\]
Thus the eight transverse components are exactly the possible terminal
internal--internal brackets
\[
L_2\otimes L_4\longrightarrow L_6,
\qquad
\Lambda^2L_3\longrightarrow L_6;
\]
the cocycle equation then determines their lower-layer metabelian defects
up to gauge and pencil motion.
\end{proposition}

\begin{proof}
Weight-zero coboundaries and the three $\psi_i$ have zero terminal
restriction.  Exact elimination gives eight complementary integral
cocycles whose terminal-restriction matrix is
\eqref{eq:terminal-certificate} and has determinant $2$; the cocycles are
in the source archive.  Thus the restriction has rank eight, and source
and target both have dimension eight.
\end{proof}

\subsection{Residual symmetry and the formal graded slice}
\label{sec:formal-moduli}

We determine the effective residual symmetry and the resulting formal
orbit problem at $\fg_\circ$.

Put
\[
\sigma=\begin{pmatrix}-1&0\\0&1\end{pmatrix},\qquad
\tau=\begin{pmatrix}0&1\\1&0\end{pmatrix}.
\]
Thus $\sigma$ sends $X$ to $-X$, while $\tau$ interchanges $X$ and $Y$.

After a linear change of variables, $F_\circ$ is Wall's $t_2$-pencil, whose
order-four projective symmetry is $D_4\simeq V_4$
\cite[pp.~211--212, Table~4]{Wall1998}.  The next result verifies that
there is no symmetry enhancement and determines the precise graded lift
used in the local moduli problem.  Its induced tangent representation is
computed in Theorem~\ref{thm:residual-representation}.

\begin{theorem}[residual stabilizer]
\label{thm:exact-stabilizer}
The projective pencil stabilizer is
\[
\operatorname{Stab}_{\PGL(U)}(F_\circ)
=\langle[\sigma],[\tau]\rangle\simeq V_4.
\]
Its inverse image in $\GL(U)$ is
\[
G_U=\operatorname{Stab}_{\GL(U)}(F_\circ)
=\{\lambda A:\lambda\in\mathbb G_m,
A\in\{I,\sigma,\tau,\sigma\tau\}\}
\simeq(\mathbb G_m\times D_8)/\mu_2,
\]
where $D_8=\langle\sigma,\tau\rangle$ has order eight and the kernel
$\mu_2$ is
$\{(1,I),(-1,-I)\}\subset\mathbb G_m\times D_8$.  The central extension
\[
1\longrightarrow\mathbb G_m\longrightarrow G_U
\longrightarrow V_4\longrightarrow1
\]
does not split.  Under the graded lift of
Theorem~\ref{thm:graded-aut}, the scalar subgroup becomes the Carnot torus
and acts trivially on every weight-zero cochain.
\end{theorem}

\begin{proof}
In the affine coordinate $z=S/T$, the pencil defines
\[
r(z)=\frac{4z(z^2+1)}{z^4+1}.
\]
A stabilizing source transformation induces $h$ with
$r\circ g=h\circ r$.  The discriminant
\[
256(a^2-16b^2)(a^2+2b^2)^2
\]
forces $h$ to preserve separately the simple pair $\{4,-4\}$ and the
double pair $\{\beta,-\beta\}$, where $\beta^2=-2$.  After sending the
first pair to $\{0,\infty\}$, the second is $\{q,q^{-1}\}$ with
$q^4\ne1$; hence the induced target group has order at most two.

Writing $w=z+z^{-1}$ gives
\[
r=\frac{4w}{w^2-2}.
\]
Any deck transformation sends $w$ to $w$ or $-2/w$; the latter would
require a quadratic over $\overline\kk(z)$ whose discriminant is
proportional to the nonsquare $z^4+z^2+1$.  Thus the deck group has order
two.  The maps $z\mapsto z^{-1}$ and $z\mapsto-z$ attain the bound, giving
$V_4$.  In characteristic zero the stabilizer is reduced, and these four
points are defined over $\kk$.

Finally, $\sigma\tau=-\tau\sigma$ and $(\sigma\tau)^2=-I$ give the stated
nonsplit central product.  Scalars act on $L_j$ by $\lambda^j$ and hence
trivially on weight-zero cochains.
\end{proof}

We next formulate the local graded deformation problem.  Fix the graded
vector space
\[
V=\bigoplus_{i=1}^6L_i
\]
and let
\[
\mathcal X=\operatorname{Lie}^{\mathrm{gr}}(V),\qquad
\mathcal H=\prod_{i=1}^6\GL(L_i),
\]
where $\mathcal X$ is the affine scheme of degree-zero Lie laws and
$\mathcal H$ acts by change of basis.  The Carnot torus
\[
T_{\mathrm{Car}}
=\{(\lambda^i\operatorname{id}_{L_i})_{i=1}^6:
\lambda\in\mathbb G_m\}
\subset\mathcal H
\]
acts trivially on $\mathcal X$.  Put
$\overline{\mathcal H}=\mathcal H/T_{\mathrm{Car}}$.  Restriction to
$L_1$ and the graded lift of Theorem~\ref{thm:graded-aut} identify $G_U$
with the stabilizer of $\mu$ in $\mathcal H$.  Thus the stabilizers in
$\mathcal H$ and $\overline{\mathcal H}$ are
\[
G=\operatorname{Stab}_{\mathcal H}(\mu),\qquad
K=\operatorname{Stab}_{\overline{\mathcal H}}(\mu)\simeq V_4,
\]
and $G/T_{\mathrm{Car}}\simeq K$.

Let $\mathsf{Art}_{\kk}$ be the category of local Artin $\kk$-algebras
with residue field $\kk$.  Framed equivalences reduce to the identity in
$\overline{\mathcal H}$; unframed equivalences may reduce to $K$.

\begin{theorem}[formal graded deformation groupoids]
\label{thm:formal-graded-hull}
The graded law scheme $\mathcal X$ is smooth of dimension $48$ at $\mu$.
There is a $K$-stable formal slice
\[
S\simeq\operatorname{Spf}\kk[[t_1,\ldots,t_{11}]],
\qquad T_\mu S\simeq H^2_0(\fg_\circ;\fg_\circ),
\]
such that the effective framed and unframed deformation groupoids over
$A\in\mathsf{Art}_{\kk}$ are respectively
\[
S(A)\qquad\text{and}\qquad[S(A)/K].
\]
Thus the effective unframed formal stack is $[S/V_4]$.  The slice admits
$K$-equivariant coordinates in which this action is linear.  For the full
change-of-basis group, $[S/G]\to[S/K]$ is the nonneutral
$\mathbb G_m$-gerbe defined by the nonsplit extension
$1\to\mathbb G_m\to G\to K\to1$.
\end{theorem}

\begin{proof}
The degree-zero Jacobi map has source $C^2_0$ of dimension $73$, target
$C^3_0$ of dimension $25$, and derivative $d_2$ of rank $25$.  Hence
$\mathcal X$ is smooth of dimension $48$ at $\mu$, with tangent space
$Z^2_0$.  The effective group $\overline{\mathcal H}$ has dimension $37$,
its orbit tangent is $B^2_0$, and its stabilizer is $K\simeq V_4$.

Choose a $K$-stable complement $Z^2_0=B^2_0\oplus N$ by averaging.  Lifting
$N^*$ to $K$-stable formal parameters cuts out a smooth formal slice $S$.
The derivative of the action map
\[
\widehat{\overline{\mathcal H}}_e\times S
\longrightarrow\widehat{\mathcal X}_\mu
\]
is an isomorphism, so the formal inverse-function theorem makes the map an
isomorphism.  This gives the two groupoids and
$T_\mu S\simeq Z^2_0/B^2_0=H^2_0$.  Averaging a lift of the cotangent space
linearizes the finite $K$-action.  Finally, the full stabilizer is $G$ and
the Carnot torus acts trivially; the last assertion follows from
Theorem~\ref{thm:exact-stabilizer}.
\end{proof}

\begin{remark}[IA automorphisms]
An IA automorphism generally mixes weight zero with positive weights, so
$\operatorname{IAut}(\fg_\circ)$ does not act on the weight-zero
deformation functor and is not part of the residual gauge group.  Its
identity action on the associated graded does not change this conclusion.
\end{remark}

\subsection{Tangent representation and the terminal bracket map}

Let $\chi_{\epsilon,\delta}$ denote the character on which $\sigma$ and
$\tau$ act by $\epsilon$ and $\delta$, respectively.
For $\phi\in E_{\mathrm{top}}$, define terminal coordinates
$v_1,\ldots,v_8$ by
\[
\begin{aligned}
\phi(Z_0,Z_{x^2})&=v_1T_0+v_2T_1,&
\phi(Z_0,Z_{xy})&=v_3T_0+v_4T_1,\\
\phi(Z_0,Z_{y^2})&=v_5T_0+v_6T_1,&
\phi(Z_x,Z_y)&=v_7T_0+v_8T_1.
\end{aligned}
\]

\begin{theorem}[residual tangent representation]
\label{thm:residual-representation}
In the pencil-motion basis of \eqref{eq:psi-basis},
\[
\rho_P(\sigma)=\operatorname{diag}(1,-1,1),\qquad
\rho_P(\tau)=\operatorname{diag}(-1,-1,1),
\]
and hence
\[
P\simeq\chi_{+,-}\oplus\chi_{-,-}\oplus\chi_{+,+}.
\]
In this terminal coordinate order,
\[
\begin{aligned}
\rho_E(\sigma)&=\operatorname{diag}(-1,1,1,-1,-1,1,1,-1),\\
\rho_E(\tau)(v_1,\ldots,v_8)
&=(-v_5,-v_6,-v_3,-v_4,-v_1,-v_2,v_7,v_8).
\end{aligned}
\]
Consequently
\[
E_{\mathrm{top}}\simeq2\kk[V_4]
\]
and the canonical exact sequence of representations
\[
0\longrightarrow P\longrightarrow H^2_0
\longrightarrow E_{\mathrm{top}}\longrightarrow0
\]
gives
\[
H^2_0\simeq
3\chi_{+,+}\oplus3\chi_{+,-}\oplus
2\chi_{-,+}\oplus3\chi_{-,-}.
\]
The splitting used in the last display exists equivariantly but is not
canonical.
\end{theorem}

\begin{proof}
Push-forward of the terminal elementary maps gives the two displayed
matrices.  Their simultaneous eigenspaces have bases
\[
\begin{array}{c|c}
(+,+)&v_2-v_6,\ v_7\\
(+,-)&v_2+v_6,\ v_3\\
(-,+)&v_1-v_5,\ v_8\\
(-,-)&v_1+v_5,\ v_4,
\end{array}
\]
so every character occurs twice.
Thus $E_{\mathrm{top}}\simeq2\kk[V_4]$; semisimplicity and the displayed
action on $P$ give the asserted decomposition of $H^2_0$.
\end{proof}

The terminal bracket also defines nonlinear loci in the formal base.  Put
\[
W_{\mathrm{term}}=(L_2\otimes L_4)\oplus\Lambda^2L_3,
\qquad
E_{\mathrm{top}}=\Hom(W_{\mathrm{term}},L_6),
\]
which is identified with $\operatorname{Mat}_{2\times4}$ after the basis
choices above.  Write $\widehat E_{\mathrm{top},0}$ for the formal
completion of this affine space at its zero tensor.

\begin{theorem}[terminal bracket map]
\label{thm:terminal-map}
Write $S=\operatorname{Spf}R_S$.  On a miniversal slice there is a
$K$-equivariant formal terminal-bracket map
\[
\Theta:S\longrightarrow\widehat E_{\mathrm{top},0}
\]
whose derivative at $\mu$ is $\Delta_{\mathrm{top}}$.
It is formally smooth of relative dimension three and can be straightened
noncanonically to the projection
\[
\operatorname{Spf}\kk[[p_1,p_2,p_3,e_1,\ldots,e_8]]
\longrightarrow\operatorname{Spf}\kk[[e_1,\ldots,e_8]].
\]
Its zero fiber is the scheme-theoretic local metabelian locus.  The affine
pencil-motion family induces a formal isomorphism
$\widehat{\mathbb A}^3_0\simeq\Theta^{-1}(0)$; thus this zero fiber is the
local binary formation germ on the slice.
\end{theorem}

\begin{proof}
For a graded law $\nu$, let $\Theta(\nu)$ be the restriction of its bracket
to
$(L_2\otimes L_4)\oplus\Lambda^2L_3$.  Under
$T_\mu S\simeq H^2_0$, its differential is $\Delta_{\mathrm{top}}$, which
is surjective by
Proposition~\ref{prop:terminal-defect}.  Formal smoothness and straightening
follow; the target is completed at the zero terminal bracket because the
deformation parameters are topologically nilpotent.

It remains to identify the zero fiber.  The bracket maps
$L_1\otimes L_j\to L_{j+1}$, with $1\le j\le5$ and $\Lambda^2L_1$ in
place of $L_1\otimes L_1$, are surjective at $\mu$.  Chosen maximal
minors remain units in $R_S$.  Thus the derived algebra
over every Artin base is $\bigoplus_{j\ge2}L_j$.  Since $\dim L_2=1$, one
also has $[L_2,L_2]=0$.  If $\Theta(\nu)=0$, then
$[L_2,L_4]=[L_3,L_3]=0$.  Jacobi gives
\[
[L_1,[L_2,L_3]]=0.
\]
At $\mu$, the joint action map
\[
L_5\longrightarrow\Hom(L_1,L_6),\qquad
z\longmapsto(x\mapsto[x,z]),
\]
is an isomorphism, and its determinant remains a unit in the formal
neighborhood.  Hence $[L_2,L_3]=0$ as well, so the derived algebra is
abelian.  Since the argument inverts a unit matrix in the quotient by the
terminal entries, it identifies the closed deformation subfunctors
scheme-theoretically.  Conversely, if the derived algebra is abelian, both
terminal restrictions vanish.  The
affine pencil family $\mu+\sum p_i\psi_i$ has a classifying map into this
smooth $3$-dimensional zero fiber, and that map induces an isomorphism on
tangent spaces.  The ordinary formal inverse-function theorem therefore
identifies the zero fiber with the pencil-motion parameter germ.
\end{proof}

\begin{corollary}[terminal rank loci]
\label{cor:terminal-rank-loci}
In the straightened coordinates of Theorem~\ref{thm:terminal-map},
$\Theta$ is a generic $2\times4$ matrix.  Its exact-rank $0,1,2$ loci
have dimensions $3,8,11$.  The rank-at-most-one locus is irreducible of
codimension three and is singular precisely along the rank-zero locus,
which is the scheme-theoretic metabelian locus.  These loci are
$V_4$-stable.
\end{corollary}

\begin{proof}
Apply the standard $2\times4$ determinantal calculation to the projection
form of $\Theta$.
\end{proof}

\subsection{Primary obstructions}
\label{sec:primary-obstructions}

The vanishing $H^3_0=0$ removes weight-zero obstructions.  The cohomology
table leaves possible primary obstructions in negative weights, which we
now compute.

The Nijenhuis--Richardson bracket \cite{NijenhuisRichardson1966,
NijenhuisRichardson1967} induces graded-symmetric pairings
\[
H^2_a(\fg_\circ;\fg_\circ)\otimes
H^2_b(\fg_\circ;\fg_\circ)
\longrightarrow H^3_{a+b}(\fg_\circ;\fg_\circ).
\]
The primary obstruction to an infinitesimal class $[\phi]$ is
\[
\operatorname{ob}([\phi])
=\left[\frac12[\phi,\phi]_{\NR}\right].
\]

\begin{theorem}[surjective negative-weight obstruction pairings]
\label{thm:obstruction-ranks}
For $\fg_\circ$, the following four maps are surjective:
\[
\begin{array}{rcll}
\Sym^2H^2_{-2}&\longrightarrow&H^3_{-4},&\rank=39,\\
H^2_{-2}\otimes H^2_{-1}&\longrightarrow&H^3_{-3},&\rank=12,\\
H^2_{-2}\otimes H^2_0&\longrightarrow&H^3_{-2},&\rank=2,\\
\Sym^2H^2_{-1}&\longrightarrow&H^3_{-2},&\rank=2.
\end{array}
\]
In particular, the negative-weight deformation problem is not
unobstructed.
\end{theorem}

\begin{proof}
Over each of $\mathbb F_{1\,000\,003}$ and
$\mathbb F_{1\,000\,033}$, the program in
Table~\ref{tab:finite-certificates} constructs the induced obstruction
maps after reducing the integral bracket.  The relevant
Chevalley--Eilenberg ranks agree with their exact rational values, so
cohomology commutes with reduction at these primes.  The four induced
matrices have full row ranks $39,12,2,2$, respectively.  Rank cannot
increase under reduction from characteristic zero; hence the four
characteristic-zero maps are surjective.  Equation
\eqref{eq:CE-base-change} gives the result over $\kk$.
\end{proof}

\begin{corollary}[nonnegative-weight unobstructedness]
\label{cor:nonnegative-unobstructed}
Let $\mathcal C^i=C^{i+1}(\fg_\circ;\fg_\circ)$ be the shifted
Nijenhuis--Richardson DGLA.  Its sub-DGLA of cochains of Carnot weight at
least zero satisfies $H^2(\mathcal C_{\ge0})=0$.  Its formal deformation
functor is unobstructed.
\end{corollary}

\begin{proof}
Weights add under the bracket, and the obstruction space is
\[
H^2(\mathcal C_{\ge0})
=\bigoplus_{w\ge0}H^3_w(\fg_\circ;\fg_\circ)=0
\]
by Theorem~\ref{thm:cohomology-table}.
\end{proof}

\medskip
\noindent
Taken together, these results identify one mechanism behind the general
and computed parts of the paper.  Formation turns graded metabelian Lie
laws into commutative Quot data, while derived--derived restriction
detects infinitesimal directions leaving that locus.  Intrinsic
lower-central tensors recover the full inverse-system module, and the
rank-three failure of scalar faithfulness shows why the parameters must
remain module-valued.  For the quartic pencil, this mechanism is realized
by a formally smooth map from the $11$-dimensional graded germ to an
$8$-dimensional terminal-bracket space, with the $3$-dimensional
binary-pencil germ as its scheme-theoretic zero fiber.  The residual
symmetry and primary obstructions describe the surrounding graded
deformation problem.

\appendix

\section{Computational verification}
\label{app:verification}

All matrices below are constructed from the displayed integral bracket in
the homogeneous basis fixed in Section~\ref{sec:quartic-pencil}.  Put
$c^q_w=\dim C^q_w$ and
$r^q_w=\rank(d:C^q_w\to C^{q+1}_w)$.  Fraction-free elimination over
$\mathbb Q$ gives the following compact rank certificate for every weight
in which $H^2$ or $H^3$ is nonzero.

\begin{table}[H]
\centering
\caption{Exact Chevalley--Eilenberg rank certificate.  All omitted
weights with nonzero cochain blocks are exact.}
\label{tab:ce-certificate}
\small
\setlength{\tabcolsep}{5pt}
\begin{tabular}{@{}rccc@{}}
\toprule
$w$ & $(c^1_w,c^2_w,c^3_w)$ & $(r^1_w,r^2_w,r^3_w)$
& $(\dim H^2_w,\dim H^3_w)$\\
\midrule
$-7$ & $(0,94,632)$  & $(0,94,504)$  & $(0,34)$\\
$-6$ & $(0,143,559)$ & $(0,143,356)$ & $(0,60)$\\
$-5$ & $(4,176,452)$ & $(4,172,218)$ & $(0,62)$\\
$-4$ & $(10,183,324)$& $(10,173,112)$& $(0,39)$\\
$-3$ & $(14,166,210)$& $(14,152,46)$ & $(0,12)$\\
$-2$ & $(21,142,120)$& $(21,102,16)$ & $(19,2)$\\
$-1$ & $(30,112,60)$ & $(30,56,4)$   & $(26,0)$\\
$0$  & $(38,73,25)$  & $(37,25,0)$   & $(11,0)$\\
$1$  & $(30,40,8)$   & $(28,8,0)$    & $(4,0)$\\
\bottomrule
\end{tabular}
\end{table}

The remaining computer-assisted statements reduce to the finite
certificates in Table~\ref{tab:finite-certificates}.  A matrix size
$a\times b$ means $a$ rows and $b$ columns.

\begin{table}[H]
\centering
\caption{Finite certificates for the computed application.}
\label{tab:finite-certificates}
\small
\setlength{\tabcolsep}{4pt}
\begin{tabular}{@{}>{\raggedright\arraybackslash}p{0.17\textwidth}
>{\raggedright\arraybackslash}p{0.20\textwidth}
>{\raggedright\arraybackslash}p{0.29\textwidth}
>{\raggedright\arraybackslash}p{0.23\textwidth}@{}}
\toprule
Assertion & Matrix or data & Exact certificate & Archive program\\
\midrule
Weighted cohomology
& $d_q(w)$ for $-16\le w\le4$; at $w=0$, sizes $73\times38$ and
$25\times73$
& $d_{q+1}d_q=0$; ranks $37,25$ and pivot minors $-48,1$
& \path{compute_ce.py}; \path{check_ce_exact.py}\\[2pt]

Pencil-motion tangent
& $73\times41$ matrix $[d_1\mid\psi_0\mid\psi_1\mid\psi_2]$
& Successive ranks $37,38,39,40$; a $40$-minor is $-192$
& \path{verify_polynomial_slice.py}\\[2pt]

Weight-zero basis and terminal defect
& $73\times49$ augmented matrix and $8\times8$ restriction matrix
& Rank $48$ with a $48$-minor $-1536$; terminal determinant $2$
& \path{compute_h2_zero_basis.py};
\path{verify_metabelian_tangent.py}\\[2pt]

Residual $V_4$-action
& Signed $3\times3$, $8\times8$, and $11\times11$ matrices
& $\sigma^2=\tau^2=(\sigma\tau)^2=1$ and the character multiplicities of
Theorem~\ref{thm:residual-representation}: $(3,3,2,3)$
& \path{stabilizer_terminal_representation.py}\\[2pt]

Primary obstruction maps
& $39\times190$, $12\times494$, $2\times209$, $2\times351$
& Ranks $39,12,2,2$ modulo both $1\,000\,003$ and $1\,000\,033$
& \path{compute_obstructions.py}\\
\bottomrule
\end{tabular}
\end{table}

For completeness, order the terminal coordinates by
\[
(Z_0,Z_{x^2}),\ (Z_0,Z_{xy}),\ (Z_0,Z_{y^2}),\ (Z_x,Z_y),
\qquad\text{each followed by }T_0,T_1.
\]
The columns of the eight complementary cocycles then restrict to
\begin{equation}\label{eq:terminal-certificate}
\left(
\begin{array}{rrrrrrrr}
0&0&0&0&-1&0&0&0\\
-1&0&0&0&0&0&0&0\\
0&0&0&0&0&0&1&1\\
0&0&0&-1&0&0&0&0\\
0&0&0&0&0&1&0&0\\
0&1&0&0&0&0&0&0\\
0&0&0&0&0&0&-1&1\\
0&0&1&1&0&0&0&0
\end{array}
\right),
\qquad \det=2.
\end{equation}

For the modular rows, the required cohomology ranks first agree with their
rational values, so cohomology commutes with reduction at both good
primes.  Full target rank after reduction then forces full target rank in
characteristic zero.  The source archive
\cite{BlattnerCode2026} contains the construction programs, exact block
ranks, sparse cocycles, commands, and the SHA-256 manifest
\path{COMPUTATION_CERTIFICATE.md}.  Five independently sampled regular
integer pencils reproduce the complete $H^2$ table; this is a robustness
check, not part of the proof of Proposition~\ref{prop:generic-H20}.
\begingroup
\renewcommand{\bibliofont}{\small}

\section*{Funding}
This research did not receive any specific grant from funding
agencies in the public, commercial, or not-for-profit sectors.

\section*{Declaration of competing interest}
The author declares that they have no known competing financial interests or
personal relationships that could have appeared to influence the work
reported in this paper.%

\section*{Data availability}
The verification programs and the SHA-256 manifest described in
Appendix~\ref{app:verification} accompany this submission as
supplementary material and are permanently archived at
\url{https://doi.org/10.5281/zenodo.21768573}.

\section*{Declaration of generative AI and AI-assisted technologies in the
manuscript preparation process}

During the preparation of this work, the author used Claude (Anthropic) for drafting and editorial feedback.  After using
these tools, the author reviewed and verified the content and takes full
responsibility for the published article.

\bibliographystyle{amsplain}
\bibliography{references}

@article{DrenskyFindik2012,
  author  = {Drensky, Vesselin and F{\i}nd{\i}k, {\c{S}}ehmus},
  title   = {Inner and Outer Automorphisms of Free Metabelian Nilpotent {L}ie Algebras},
  journal = {Communications in Algebra},
  volume  = {40},
  number  = {12},
  year    = {2012},
  pages   = {4389--4403},
  doi     = {10.1080/00927872.2011.610071},
  eprint  = {1003.0350},
  archivePrefix = {arXiv},
  primaryClass  = {math.RA}
}

@article{GalitskiTimashev1999,
  author  = {Galitski, L. Yu. and Timashev, D. A.},
  title   = {On the Classification of Metabelian {L}ie Algebras},
  journal = {Journal of Lie Theory},
  volume  = {9},
  number  = {1},
  year    = {1999},
  pages   = {125--156},
  doi     = {10.5802/jolt.163}
}

@article{Gauger1973,
  author  = {Gauger, Michael A.},
  title   = {On the Classification of Metabelian {L}ie Algebras},
  journal = {Transactions of the American Mathematical Society},
  volume  = {179},
  year    = {1973},
  pages   = {293--329},
  doi     = {10.1090/S0002-9947-1973-0325719-0}
}

@article{GavioliMontiYoung2001,
  author  = {Gavioli, Norberto and Monti, Valerio and Young, David S.},
  title   = {Metabelian Thin {L}ie Algebras},
  journal = {Journal of Algebra},
  volume  = {241},
  number  = {1},
  year    = {2001},
  pages   = {102--117},
  doi     = {10.1006/jabr.2001.8752}
}

@article{Iarrobino2004,
  author  = {Iarrobino, Anthony},
  title   = {Ancestor Ideals of Vector Spaces of Forms, and Level Algebras},
  journal = {Journal of Algebra},
  volume  = {272},
  number  = {2},
  year    = {2004},
  pages   = {530--580},
  doi     = {10.1016/S0021-8693(03)00425-3}
}

@article{HarimaMiglioreNagelWatanabe2003,
  author  = {Harima, Tadahito and Migliore, Juan C. and Nagel, Uwe and Watanabe, Junzo},
  title   = {The weak and strong {L}efschetz properties for {A}rtinian {$K$}-algebras},
  journal = {Journal of Algebra},
  volume  = {262},
  number  = {1},
  year    = {2003},
  pages   = {99--126},
  doi     = {10.1016/S0021-8693(03)00038-3}
}

@article{Kapranov2012,
  author  = {Kapranov, Mikhail},
  title   = {Real Mixed {H}odge Structures},
  journal = {Journal of Noncommutative Geometry},
  volume  = {6},
  number  = {2},
  year    = {2012},
  pages   = {321--342},
  doi     = {10.4171/JNCG/93},
  eprint  = {0802.0215},
  archivePrefix = {arXiv},
  primaryClass  = {math.AG}
}

@misc{Kapranov2015,
  author        = {Kapranov, Mikhail},
  title         = {Membranes and Higher Groupoids},
  year          = {2015},
  eprint        = {1502.06166},
  archivePrefix = {arXiv},
  primaryClass  = {math.DG},
  note          = {Preprint, arXiv:1502.06166},
  url           = {https://arxiv.org/abs/1502.06166}
}

@incollection{KleimanKleppe2025,
  author        = {Kleiman, Steven L. and Kleppe, Jan O.},
  title         = {Macaulay Duality and Its Geometry},
  booktitle     = {Perspectives on Four Decades of Algebraic Geometry, Volume 1},
  series        = {Progress in Mathematics},
  volume        = {351},
  year          = {2025},
  pages         = {337--427},
  publisher     = {Birkh{\"a}user},
  address       = {Cham},
  doi           = {10.1007/978-3-031-66230-0_13},
  eprint        = {2210.10934},
  archivePrefix = {arXiv},
  primaryClass  = {math.AG}
}

@article{LegerLuks1972,
  author  = {Leger, G. and Luks, E.},
  title   = {On a Duality for Metabelian {L}ie Algebras},
  journal = {Journal of Algebra},
  volume  = {21},
  number  = {2},
  year    = {1972},
  pages   = {266--270},
  doi     = {10.1016/0021-8693(72)90021-X}
}

@article{NijenhuisRichardson1966,
  author  = {Nijenhuis, Albert and Richardson, Roger W.},
  title   = {Cohomology and Deformations in Graded {L}ie Algebras},
  journal = {Bulletin of the American Mathematical Society},
  volume  = {72},
  year    = {1966},
  pages   = {1--29}
}

@article{NijenhuisRichardson1967,
  author  = {Nijenhuis, Albert and Richardson, Roger W.},
  title   = {Deformations of {L}ie Algebra Structures},
  journal = {Journal of Mathematics and Mechanics},
  volume  = {17},
  number  = {1},
  year    = {1967},
  pages   = {89--105}
}

@article{PapadimaSuciu2004,
  author        = {Papadima, Stefan and Suciu, Alexander I.},
  title         = {Chen {L}ie Algebras},
  journal       = {International Mathematics Research Notices},
  volume        = {2004},
  number        = {21},
  year          = {2004},
  pages         = {1057--1086},
  doi           = {10.1155/S1073792804132017},
  eprint        = {math/0307087},
  archivePrefix = {arXiv},
  primaryClass  = {math.GR}
}

@article{Papistas1992,
  author  = {Papistas, Athanassios I.},
  title   = {On Automorphism Groups of 2-Generator Metabelian {L}ie Algebras},
  journal = {Communications in Algebra},
  volume  = {20},
  number  = {7},
  year    = {1992},
  pages   = {1937--1953},
  doi     = {10.1080/00927879208824441}
}

@article{Papistas1996,
  author  = {Papistas, Athanassios I.},
  title   = {{IA}-Automorphisms of 2-Generator Metabelian {L}ie Algebras},
  journal = {Algebra Colloquium},
  volume  = {3},
  number  = {3},
  year    = {1996},
  pages   = {193--198}
}

@book{Reutenauer1993,
  author    = {Reutenauer, Christophe},
  title     = {Free {L}ie Algebras},
  series    = {London Mathematical Society Monographs, New Series},
  volume    = {7},
  publisher = {Oxford University Press},
  year      = {1993}
}

@misc{Suciu2026,
  author        = {Suciu, Alexander I.},
  title         = {Koszul Modules, Holonomy {L}ie Algebras, and Resonance
                   of Groups and {CDGA}s},
  year          = {2026},
  eprint        = {2604.24986},
  archivePrefix = {arXiv},
  primaryClass  = {math.AT},
  doi           = {10.48550/arXiv.2604.24986},
  url           = {https://arxiv.org/abs/2604.24986}
}

@article{Wall1998,
  author  = {Wall, C. T. C.},
  title   = {Pencils of Binary Quartics},
  journal = {Rendiconti del Seminario Matematico della Universit{\`a} di Padova},
  volume  = {99},
  year    = {1998},
  pages   = {197--217},
  url     = {https://www.numdam.org/item/RSMUP_1998__99__197_0/}
}

@article{BarrionuevoTiraoSulca2023,
  author  = {Barrionuevo, Josefina and Tirao, Paulo and Sulca, Diego},
  title   = {Deformations and Rigidity in Varieties of {L}ie Algebras},
  journal = {Journal of Pure and Applied Algebra},
  volume  = {227},
  number  = {3},
  year    = {2023},
  pages   = {107217},
  doi     = {10.1016/j.jpaa.2022.107217},
  eprint  = {2101.09670},
  archivePrefix = {arXiv}
}

@misc{BlattnerCode2026,
  author       = {Blattner, Marcel},
  title        = {Verification Programs for ``{Q}uot-{S}tack Moduli and
                  Transverse Deformations of Graded Metabelian {L}ie Algebras''},
  howpublished = {Version 1.1 [software], Zenodo},
  year         = {2026},
  doi          = {10.5281/zenodo.21768573},
  url          = {https://doi.org/10.5281/zenodo.21768573},
  note         = {\url{https://doi.org/10.5281/zenodo.21768573}}
}
\endgroup

\end{document}